\documentclass[12pt,reqno,a4paper]{amsart}
\usepackage[margin=.645in]{geometry}
\usepackage[usenames]{color}
\usepackage{amsmath,amsthm,pdfsync,verbatim,graphicx,epstopdf,enumerate}
\usepackage{placeins}
\usepackage{amsfonts}
\usepackage{amssymb}
\usepackage{cite}
\usepackage{mathrsfs}
\usepackage{enumitem}
\usepackage[colorlinks]{hyperref}
\usepackage{multicol}
\usepackage{array}
\usepackage{comment}
\usepackage{mathtools}
\mathtoolsset{showonlyrefs}
\usepackage[abbrev,nobysame]{amsrefs}
\usepackage{float}
\usepackage[framemethod=tikz]{mdframed}
\hypersetup{allcolors=blue}
\allowdisplaybreaks
\numberwithin{equation}{section}

\newcommand{\Beq}{\begin{equation}}
	\newcommand{\Eeq}{\end{equation}}
\newcommand{\beq}{\begin{equation*}}
	\newcommand{\eeq}{\end{equation*}}
\newcommand{\bal}{\begin{align}}
	\newcommand{\eal}{\end{align}}

\usepackage{mathtools}

\usepackage[notref,notcite]{}

\newcommand{\bp}{\begin{prob}}
	\newcommand{\ep}{\end{prob}}
\newcommand{\bpr}{\begin{proof}}
	\newcommand{\epr}{\end{proof}}

\newcommand{\bel}[1]{\begin{equation}\label{#1}}
	\newcommand{\ee}{\end{equation}}

\newtheorem{theorem}{Theorem}[section]

\theoremstyle{definition}

\newtheorem{assumption}[theorem]{Assumption}
\newtheorem{remark}[theorem]{Remark}

\newtheorem{defn}[theorem]{Definition}

\usepackage[usenames]{color}
\usepackage{amsmath,pdfsync,verbatim,graphicx,epstopdf,enumerate}

\begin{document}

\title[Structurally damped semilinear evolution equation]
{Structurally damped semilinear evolution equation for positive operators on Hilbert space}
\author[Aparajita Dasgupta]{Aparajita Dasgupta}
\address{
	Aparajita Dasgupta:
	\endgraf
	Department of Mathematics
	\endgraf
	Indian Institute of Technology, Delhi, Hauz Khas
	\endgraf
	New Delhi-110016, India
	\endgraf
	{\it E-mail address:} {\rm adasgupta@maths.iitd.ac.in}
}
\author[Lalit Mohan]{Lalit Mohan}
\address{
	Lalit Mohan:
	\endgraf
	Department of Mathematics
	\endgraf
	Indian Institute of Technology, Delhi, Hauz Khas
	\endgraf
	New Delhi-110016, India
	\endgraf
	{\it E-mail address:} {\rm mohanlalit871@gmail.com}
}
\author[Abhilash Tushir]{Abhilash Tushir}
\address{
Abhilash Tushir:
	\endgraf
	Tata Institute of Fundamental Research
    \endgraf
    Centre For Applicable Mathematics
    \endgraf Bangalore, Karnataka-560065, India
	\endgraf
    {\it E-mail address:} {\rm abhilash2296@gmail.com}
}

\date{\today}

\begin{abstract}
In this study, we analyze the following semilinear damped evolution equation associated with a self-adjoint, positive operator $\mathcal{L}$ on a separable Hilbert space $\mathcal{H}$ with discrete spectrum which, reads as follows:
\begin{equation*}
	\left\{\begin{array}{l}
		u_{tt}(t)+\mathcal{L}^{\theta}u_{t}(t)+\mathcal{L}^{\sigma}u(t)
=0, \quad t>0, \\
		u(0)=u_{0}\in\mathcal{H},\quad
		 u_{t}(0)=u_{1}\in\mathcal{H},
	\end{array}\right.
\end{equation*}
under different damping conditions, including the undamped $(\theta=0)$, effectively damped $(0<2\theta<\sigma)$, critically damped $(2\theta=\sigma)$,  non-effectively damped $(\sigma<2\theta<2\sigma)$, and viscoelastic type damped $(\theta=\sigma)$. This  article is devoted to examining the decay estimates of solutions to the above linear evolution equation subject to initial Cauchy data with Sobolev regularity.  
Furthermore, we also observe loss of decay with relaxation of regularity in some case.
As an application of the decay estimates, we also demonstrate the global existence (in time) of the solution in certain cases, taking into account polynomial-type nonlinearity. 
 \end{abstract}

\maketitle
\tableofcontents
\section{Introduction}\label{intro}
The study of damped evolution equations has become an increasingly prominent area of interest due to their applicability in modeling physical systems exhibiting energy dissipation—such as viscoelastic materials, structural mechanics, and wave propagation with damping. In particular, the interplay between damping mechanisms and nonlinear effects significantly influences the qualitative behavior of solutions, such as decay, stability, and long-term dynamics.

The linear wave equation with various damping conditions has been thoroughly investigated in the Euclidean setting, and an extensive theory based on well-posedness, dispersive estimates, and long-time behavior of solutions was developed. But when one goes beyond the Euclidean framework and examines evolution equations regulated by positive operators in abstract Hilbert spaces, there are many articles for the undamped case, but only a limited number of studies can be identified for the damping cases. More recent developments include sharp decay estimates and global well-posedness in various damping regimes (e.g., \cite{Wak01, Nar12, DE16, Ruz16,mic-wave1,mic-wave2, palmieri-1,palmieri-2}), including operators on manifolds, Lie groups, and non-Euclidean settings.

Recently partial differential equations have been studied in more abstract and geometrically rich settings, such as on Lie groups, manifolds with complex topologies, or quantum mechanics. These settings naturally give rise to self-adjoint positive operators with discrete spectra, for which traditional Fourier analytic tools are not directly applicable. Moreover, in physical models such as quantum magnetism (e.g., the Landau Hamiltonian), or control theory on Lie groups, damping effects play an equally important role in the stability and energy dissipation of the system. Despite this need, most studies in the literature on damped evolution equations remain restricted to specific operators or geometries, and there is a lack of a unified operator-theoretic approach that can simultaneously handle a wide class of operators and damping structures.
This paper aims to fill that gap by developing a general framework for semilinear damped evolution equations governed by fractional powers of  self-adjoint, positive operators on Hilbert spaces.

In this paper, we undertake a rigorous analysis of semilinear evolution equations featuring structural damping modeled by fractional powers of a positive operator on a Hilbert space.
Our goal is to characterize between damping and nonlinear effects by deriving decay estimates for the linear problem and then using these results to establish global existence for the nonlinear model under sufficiently small initial Cauchy data in the possible cases. The approach relies on spectral techniques adapted to the discrete operator framework and a detailed classification of damping regimes, ranging from no damping to critical and non-effective damping scenarios.

We begin by providing the notion of fractional powers of positive operator Hilbert space to set up the stage for the main problem examined in this article.

Let $\mathcal{L}$ be a densely defined self-adjoint, linear, positive operator on a separable Hilbert space $\mathcal{H}$ with the discrete spectrum 
\begin{equation}
    \{\lambda_{\xi}\geq 0:\xi\in \mathcal{I}\},
\end{equation}
where $\mathcal{I}$ is a countable set. For any $\alpha>0$, the fractional power of $\mathcal{L}$ is defined as follows:
\begin{equation}
    \mathcal{L}^{\alpha}u(t)=\mathcal{F}_{\mathcal{L}}^{-1}(\lambda_{\xi}^{\alpha}\mathcal{F}_{\mathcal{L}}u(t,\xi)),
\end{equation}
where $\mathcal{F}_{\mathcal{L}}$ and $\mathcal{F}_{\mathcal{L}}^{-1}$ are the $\mathcal{L}$-Fourier transform and inverse $\mathcal{L}$-Fourier transform, respectively. More details will be supplied in Section \ref{sec:prelim}. \\
\textbf{Assumption on the operator $\mathcal{L}$:} The operator \(\mathcal{L} \) is positive, self-adjoint and it has a discrete spectrum \( \{ \lambda_\xi\geq 0: \xi\in\mathcal{I} \} \), and its eigenfunctions form an orthonormal basis of the Hilbert space \( \mathcal{H} \). Moreover, the spectrum doesn't have cluster points.

Considering an abstract point of view has a key advantage that the results will cover a wide range of models, including those arising in quantum mechanics, sub-Riemannian analysis etc.\\
\textbf{Examples of operator \(\mathcal{L} \):} The framework adopted in this paper covers a variety of important examples of operators \( \mathcal{L} \) with a discrete spectrum and orthonormal eigenfunctions:
\begin{itemize}
    \item the \textbf{Landau Hamiltonian} in 2D quantum mechanics;
    \item \textbf{harmonic and anharmonic oscillators} in physics;
    \item higher-dimensional \textbf{Hamiltonian operators};
    \item \textbf{sums of squares of vector fields} on compact Lie groups; and
    \item differential operators on \textbf{compact manifolds with boundary}.
\end{itemize}
Detailed analysis of such operators and their spectral properties can be found in \cite{mic-wave1,mic-wave2}. 

We consider the following fractional abstract damped wave equation associated with the operator $\mathcal{L}$ on the Hilbert space $\mathcal{H}$:
\begin{equation}\label{main:Hilbert}
	\left\{\begin{array}{l}
		u_{tt}(t)+\mathcal{L}^{\theta}u_{t}(t)+\mathcal{L}^{\sigma}u(t)
=f(u), \quad t>0, \\
		u(0)=u_{0}\in\mathcal{H},\quad
		 u_{t}(0)=u_{1}\in\mathcal{H},
	\end{array}\right.
\end{equation}
where  the exponents \( \theta \geq 0 \) and \( \sigma > 0 \) control the damping and stiffness, respectively, and $f$ satisfies the condition:
\begin{equation}
|f(u)-f(v)| \lesssim |u-v|\left(|u|^{p-1}+|v|^{p-1}\right), ~~p>1~~,\quad \text{and}\quad f(0)=0.
\end{equation}
A crucial classification introduced in \cite{classif} distinguishes between \emph{effective} and \emph{non-effective} damping. The damping is said to be:
\begin{itemize}
\item \textbf{undamped} if $\theta=0$;
    \item \textbf{effective} if \( 2\theta \in (0, \sigma) \);
     \item \textbf{critical (scale-invariant)} when \( 2\theta = \sigma \); 
    \item \textbf{non-effective} if \( 2\theta \in (\sigma, 2\sigma) \); and
    \item \textbf{viscoelastic type
damping} if $\theta=\sigma$.
\end{itemize}
The analysis of damped evolution equations with nonlinearities has deep roots in the theory of semilinear PDEs. The foundational work by H.~Fujita~\cite{Fuj66} on blow-up and global existence in semilinear heat equations introduced the concept of a critical (Fujita) exponent, determining whether small initial data yields global solutions. Similar threshold behavior is found in wave equations with damping, as shown in works such as \cite{Mat76, Gla77, Lev90}, which demonstrated how damping mechanisms can suppress nonlinear blow-up phenomena. 

This study explores all five cases, analyzing how decay rates and solution regularity differ depending on the damping effectiveness. 
The undamped case has been thoroughly investigated in both Euclidean and non-Euclidean settings; see \cite{GR2,mic-wave2,mic-wave1,palmieri-2}
The case of effective damping for decay is thoroughly investigated but mainly in Euclidean settings; see \cite{Tuan2019,DABBICCO2014,Narazaki2013}. For the critical/scale-invariant case, one can refer to \cite{PHAM2015}.
Furthermore, the authors in \cite{Ikheta2014,Abbico2022,Tuan2019a} investigate the decay estimates and long-time behavior of the solutions in the case of non-effective damping. The case of viscoelastic type damping is also considered in several settings; see \cite{Shyam:visco,Ikehata,IKEHATA2017228,LIU2021663}.


The novelty of this work lies in providing a unified and operator-theoretic framework to study semilinear evolution equations with structural damping governed by a general class of positive operators with discrete spectra. Unlike classical treatments limited to the Laplacian or bounded domains, this paper allows for more abstract settings, including operators with non-Euclidean geometry and spectral properties (e.g., the Landau Hamiltonian, sums of squares on Lie groups); see \cite{AD&LM&VK&SSM,Ruz16}.
Using an \( \mathcal{L} \)-adapted non-harmonic Fourier analysis, the authors derive decay estimates in each damping regime, revealing regularity-dependent loss in decay. Furthermore, the analysis extends to global existence results under sufficiently small Cauchy initial data, linking the linear decay behavior to nonlinear stability in a precise, quantifiable way. These results generalize and strengthen known decay and existence theorems in literature while providing tools applicable to a broad range of operators beyond traditional elliptic ones.

For the sake of reading simplicity, let’s first give a brief synopsis of the
paper’s organization. 
\begin{itemize}
\item \textbf{Section~\ref{Sec:mainresults}} states the main results, starting with decay estimates for the linear homogeneous problem under Sobolev regularity assumptions on initial data, then we will observe some loss of decay with relaxation in regularity. 

    \item \textbf{Section~\ref{sec:prelim}} develops the necessary analytical framework, including the definition of fractional powers of the operator \(\mathcal{L}\), the \(\mathcal{L}\)-Fourier transform, and function spaces adapted to \(\mathcal{L}\), such as Sobolev-type spaces \(\mathcal{H}^\alpha_{\mathcal{L}}, \alpha>0\). 
    
    \item \textbf{Section~\ref{section-4}} provides preliminary estimates that apply to both the damped and the undamped cases, and are used in the calculations that follow.
    
        \item \textbf{Section \ref{section-5}} provides the proofs of the main results along with necessary remarks.
    \item \textbf{Section \ref{section-7}} applies the linear decay estimates and a fixed-point argument to establish global (in time) existence for the semilinear problem with small initial data, accounting for nonlinear effects on long-term solution behavior.
\end{itemize}

\section{Main results}\label{Sec:mainresults} 
This section will provide the article's main results, we will start by stating the decay estimates of the linear Cauchy problem given by 
\begin{equation}\label{main:linear:Hilbert}
	\left\{\begin{array}{l}
		u_{tt}(t)+\mathcal{L}^{\theta}u_{t}(t)+\mathcal{L}^{\sigma}u(t)
		=0, \quad t>0, \\
		u(0)=u_{0}\in\mathcal{H},\quad
		u_{t}(0)=u_{1}\in\mathcal{H}.
	\end{array}\right.
\end{equation}

The following notations  will be consistently used throughout the paper for clarity and consistency.
\begin{itemize}
	\item $f \lesssim g:$\,\, There exists a positive constant $C$ such that $f \leq C g,$ where the constant will not be affected by the parameters that appear.
	\item  $\mathcal{H}^{\alpha}_{\mathcal{L}}$ : The Sobolev spaces of order $\alpha>0$ associated with the operator $\mathcal{L}$.
    \item $\|(u,v)\|_{\mathcal{D}}:$ For $\mathcal{D}=X\times Y$, where $X$ and $Y$ are normed linear space, we introduce the norm 
 \begin{equation*}
     \|(u,v)\|_{\mathcal{D}}:=\|u\|_{X}+\|v\|_{Y}.
 \end{equation*}
    \item The constant $\delta$ below may change from line to line in this paper. However, we don't continuously relabel it to prevent notational clutter. Only the presence of a positive $\delta$ independent of $\xi\in\mathcal{I}$ is necessary for our aims.
\end{itemize}
With the notational groundwork in place, we now highlight the main results of the article.
We start our analysis with the undamped case, which serves as a baseline for comparing the different damped regimes detailed later.
\begin{theorem}[\textbf{Undamped}]\label{thm:case0new}
    Let $\theta=0,~\sigma,\beta>0,$ and $k\in\mathbb{N}$. Let $u$ be the solution of the homogeneous Cauchy problem \eqref{main:linear:Hilbert}.
Then the solution $u$ satisfies the following estimates
\begin{align}
\|u(t)\|_{\mathcal{H}}&\lesssim\|u_{0}\|_{\mathcal{H}}+\|u_{1}\|_{\mathcal{H}},\nonumber\\
\|\mathcal{L}^{\beta}u(t)\|_{\mathcal{H}}&\lesssim\begin{cases}
    e^{-\delta t}(\|u_{0}\|_{\mathcal{H}^{2\beta}_{\mathcal{L}}}+\|u_{1}\|_{\mathcal{H}^{2\beta-\sigma}_{\mathcal{L}}})&\text{if } \quad 2\beta>\sigma,\\
e^{-\delta t}(\|u_{0}\|_{\mathcal{H}^{2\beta}_{\mathcal{L}}}+\|u_{1}\|_{\mathcal{H}})&\text{if } \quad2\beta\leq \sigma,
\end{cases} \quad \text{and}\nonumber\\
\|\partial_{t}^{k}u(t)\|_{\mathcal{H}}&\lesssim
e^{-\delta t}\left(\|u_{0}\|_{\mathcal{H}^{k\sigma}_{\mathcal{L}}}+\|u_{1}\|_{\mathcal{H}^{(k-1)\sigma}_{\mathcal{L}}}\right)\quad k\geq 1,
\nonumber
\end{align}
for some $\delta>0$ and for all $t>0$.
\end{theorem}
\begin{remark}
    The decay estimates derived in above theorem strengthen those presented in \cite[Proposition 2.1]{palmieri-1} in the case of compact Lie group, facilitated due to  a more refined partition of $\mathcal{I}$.
\end{remark}
\begin{remark}\label{remrk2:mthm}
    The estimate $\|u(t)\|_{\mathcal{H}}$ witness a complete absence of decay even in the presence of Sobolev data  but the explanation  is somewhat technical and arises out of the Plancherel formula which is due to the possibility of vanishing $\lambda_{\xi}$ for some $\xi$; it will become clearer after the proof is presented.
\end{remark}
The following theorem provides the decay estimates for effective damping, critical damping, non-effective damping, and viscoelastic type damping
\begin{theorem}\label{thm:case0old}
    Let $\theta>0,\sigma,\beta>0,$ and $k\in\mathbb{N}$. Let $u$ be the solution of the homogeneous Cauchy problem \eqref{main:linear:Hilbert}.
Then 
\begin{enumerate}
    \item for $2\theta\in (0,\sigma]$, the solution $u$ satisfies the following estimates: 
    \begin{align}
\|u(t)\|_{\mathcal{H}}&\lesssim\|u_{0}\|_{\mathcal{H}}+t\|u_{1}\|_{\mathcal{H}},\nonumber\\
\|\mathcal{L}^{\beta}u(t)\|_{\mathcal{H}}&\lesssim\begin{cases}
    e^{-\delta t}(\|u_{0}\|_{\mathcal{H}^{2\beta}_{\mathcal{L}}}+\|u_{1}\|_{\mathcal{H}^{2\beta-\sigma}_{\mathcal{L}}})&\text{if } 2\beta>\sigma,\\
e^{-\delta t}(\|u_{0}\|_{\mathcal{H}^{2\beta}_{\mathcal{L}}}+\|u_{1}\|_{\mathcal{H}})&\text{if } 2\beta\leq\sigma,
\end{cases}\quad \text{and}\label{case0:spacedernorm new}\\
\|\partial_{t}^{k}u(t)\|_{\mathcal{H}}&\lesssim\begin{cases}
   e^{-\delta t}\|u_{0}\|_{\mathcal{H}^{\sigma}_{\mathcal{L}}}+\|u_{1}\|_{\mathcal{H}},&\text{if } k=1,\\
 e^{-\delta t}(\|u_{0}\|_{\mathcal{H}^{k\sigma}_{\mathcal{L}}}+\|u_{1}\|_{\mathcal{H}^{(k-1)\sigma}_{\mathcal{L}}}),&\text{if } k\geq 2.
\end{cases}
\label{case0:timedersoln:k>1 new}
\end{align}
\item for $2\theta\in (\sigma,2\sigma]$, the solution $u$ satisfies the following estimates
\begin{align}
\|u(t)\|_{\mathcal{H}}&\lesssim\|u_{0}\|_{\mathcal{H}}+t\|u_{1}\|_{\mathcal{H}},\nonumber\\
\|\mathcal{L}^{\beta}u(t)\|_{\mathcal{H}}&\lesssim\begin{cases}
    e^{-\delta t}(\|u_{0}\|_{\mathcal{H}^{2\beta}_{\mathcal{L}}}+\|u_{1}\|_{\mathcal{H}^{2\beta-\sigma}_{\mathcal{L}}})&\text{if } 2\beta>\sigma,\\
e^{-\delta t}(\|u_{0}\|_{\mathcal{H}^{2\beta}_{\mathcal{L}}}+\|u_{1}\|_{\mathcal{H}})&\text{if } 2\beta\leq\sigma,
\end{cases}\quad \text{and}\label{thm5:spaceder}\\
\|\partial_{t}^{k}u(t)\|_{\mathcal{H}}&\lesssim
\begin{cases}
    e^{-\delta t}\left\|u_{0}\right\|_{\mathcal{H}^{2\theta}_{\mathcal{L}}}+\|u_{1}\|_{\mathcal{H}},&\text{if } k=1,\\
 e^{-\delta t}(\|u_{0}\|_{\mathcal{H}^{2k\theta}_{\mathcal{L}}}+\|u_{1}\|_{\mathcal{H}^{2(k-1)\theta}_{\mathcal{L}}}),&\text{if } k\geq 2,
\end{cases}\label{thm5:timeder}
\end{align}
\end{enumerate}
for some $\delta>0$ and for all $t>0$.
\end{theorem}
Due to the absence of decay in $u$ and its first time derivative, even with the initial data exhibiting Sobolev regularity, additional structural assumptions are required to obtain improved decay estimates. Specifically, if the operator $\mathcal{L}$ is presumed to be strictly positive, or if a positive mass term  is incorporated such that $\mathcal{L}+m$ becomes strictly positive, then one may definitely establish exponential decay for the solution. This adjustment substantially eradicates the impact of appearing low-frequency components, which frequently account for the lack of decay in this setting. More precisely, the following theorem provides exponential with sufficiently regular Cauchy data and strictly positive operator:
\begin{theorem}\label{thm:expdecay}
         Let $\theta,\beta\geq0,\sigma>0,$ and $k\in\mathbb{N}$. Assume that $\mathcal{L}$ be a  self-adjoint and strictly positive operator and $u$ be the solution of the homogeneous Cauchy problem \eqref{main:linear:Hilbert}.
Then \begin{enumerate}
     \item for $2\theta\in [0,\sigma]$, the solution $u$ satisfies the following estimates \begin{equation}
\|\partial_{t}^{k}\mathcal{L}^{\beta} u(t)\|_{\mathcal{H}}\lesssim e^{-\delta t}(\|u_{0}\|_{\mathcal{H}_{\mathcal{L}}^{m_{k,\beta}}}+\|u_{1}\|_{\mathcal{H}_{\mathcal{L}}^{n_{k,\beta}}}),\quad k\geq 1,\beta\geq0,\nonumber
\end{equation}
where $ m_{k,\beta}=k\sigma+2\beta$ and $n_{k,\beta}=\max\{(k-1)\sigma,0\}+\begin{cases}
        2\beta-\sigma\quad &\text{if }2\beta>\sigma,\\
        0&\text{if } 2\beta\leq\sigma,
    \end{cases}$ and
    \item for $2\theta\in (\sigma,2\sigma]$, the solution $u$ satisfies the following estimates
    \begin{equation}
\|\partial_{t}^{k}\mathcal{L}^{\beta} u(t)\|_{\mathcal{H}}\lesssim e^{-\delta t}(\|u_{0}\|_{\mathcal{H}_{\mathcal{L}}^{m_{k,\beta}}}+\|u_{1}\|_{\mathcal{H}_{\mathcal{L}}^{n_{k,\beta}}}),\quad k\geq 1,\beta\geq0,\nonumber
\end{equation}
where $ m_{k,\beta}=2k\theta+2\beta$ and $n_{k,\beta}=\max\{2(k-1)\theta,0\}+\begin{cases}
        2\beta-\sigma\quad &\text{if }2\beta>\sigma,\\
        0&\text{if } 2\beta\leq\sigma.
    \end{cases}$
 \end{enumerate} 
\end{theorem}
\begin{remark}
   The decay estimates in the aforementioned theorem are also aligned with those found in \cite{mic-wave2}, given the explanation provided before stating the above theorem.
\end{remark}
As we can observe in the Theorem \ref{thm:case0old} and \ref{thm:expdecay}, a certain degree of Sobolev regularity in the initial data is typically necessary to obtain meaningful decay estimates on the solution. 
However, in the following theorem, it is observed that even when only the minimal regularity of the initial Cauchy data is utilized, i.e., 
$u_{0},u_{1}\in\mathcal{H}$, some decay can still be obtained in the cases of effective damping, critical damping, and non-effective damping.
 Let's begin with the case of effective damping:
\begin{theorem}[\textbf{Effective damping}]\label{thm:case1}
   Let $2\theta\in (0,\sigma)$, $\beta>0$, and $k\in\mathbb{N}$. Let $u$ be the solution of the following homogeneous Cauchy problem \eqref{main:linear:Hilbert},
then the solution $u$ satisfies the following estimates
\begin{align}
\|u(t)\|_{\mathcal{H}}&\lesssim\|u_{0}\|_{\mathcal{H}}+t\|u_{1}\|_{\mathcal{H}},\nonumber\\
\|\mathcal{L}^{\beta}u\|_{\mathcal{H}}&\lesssim \begin{cases}
        t^{-\frac{\beta}{\theta}}\|u_{0}\|_{\mathcal{H}}+t^{-\frac{2\beta-\sigma}{2\theta}}\|u_{1}\|_{\mathcal{H}}&\text{if } t\in (0,1],\\
        e^{-\delta t}\left(\|u_{0}\|_{\mathcal{H}}+\|u_{1}\|_{\mathcal{H}}\right)&\text{if } t\in [1,\infty),
    \end{cases}\quad \text{when} \quad 2\beta>\sigma,\nonumber\\
    \|\mathcal{L}^{\beta}u\|_{\mathcal{H}}&\lesssim \begin{cases}
        t^{-\frac{\beta}{\theta}}\|u_{0}\|_{\mathcal{H}}+e^{-\delta t}\|u_{1}\|_{\mathcal{H}}&\text{if } t\in (0,1],\\
        e^{-\delta t}\left(\|u_{0}\|_{\mathcal{H}}+\|u_{1}\|_{\mathcal{H}}\right)&\text{if } t\in [1,\infty),
    \end{cases}\quad \text{when} \quad 2\beta\leq\sigma,\nonumber\\
      \|\partial_{t}u\|_{\mathcal{H}}&\lesssim \begin{cases}
        t^{-\frac{\sigma}{2\theta}}\|u_{0}\|_{\mathcal{H}}+\|u_{1}\|_{\mathcal{H}}&\text{if } t\in (0,1],\\
        e^{-\delta t}\|u_{0}\|_{\mathcal{H}}+\|u_{1}\|_{\mathcal{H}}&\text{if } t\in [1,\infty),
    \end{cases}\quad \text{and}\nonumber\\
    \|\partial_{t}^{k}u(t)\|_{\mathcal{H}}&\lesssim
\begin{cases}
 t^{-\frac{k\sigma}{2\theta}}\|u_{0}\|_{\mathcal{H}}+t^{-\frac{(k-1)\sigma}{2\theta}}\|u_{1}\|_{\mathcal{H}}&\text{if } t\in (0,1],\\
   e^{-\delta t}\left(\|u_{0}\|_{\mathcal{H}}+\|u_{1}\|_{\mathcal{H}}\right)&\text{if } t\in [1,\infty),
\end{cases}\quad \text{when} \quad k\geq 2,\nonumber
\end{align}    
for some $\delta>0$ and for all $t>0$. 
\end{theorem}
The following theorem provides the decay estimates in the case of critical damping.
\begin{theorem}[\textbf{Critical damping}]\label{thm:case2}
Let $2\theta = \sigma$, $\beta>0$, and $k\in\mathbb{N}$. Let $u$ be the solution of the following homogeneous Cauchy problem \eqref{main:linear:Hilbert},
then the solution $u$ satisfies the following estimates
\begin{align}
\|u(t)\|_{\mathcal{H}}&\lesssim\|u_{0}\|_{\mathcal{H}}+t\|u_{1}\|_{\mathcal{H}},\nonumber\\
\|\mathcal{L}^{\beta}u\|_{\mathcal{H}}&\lesssim \begin{cases}
        t^{-\frac{\beta}{\theta}}\|u_{0}\|_{\mathcal{H}}+t^{-\frac{\beta-\theta}{\theta}}\|u_{1}\|_{\mathcal{H}}&\text{if } t\in (0,1],\\
        e^{-\delta t}\left(\|u_{0}\|_{\mathcal{H}}+\|u_{1}\|_{\mathcal{H}}\right)&\text{if } t\in [1,\infty),
    \end{cases}\quad \text{when} \quad \beta>\theta,\nonumber\\
    \|\mathcal{L}^{\beta}u\|_{\mathcal{H}}&\lesssim \begin{cases}
        t^{-\frac{\beta}{\theta}}\|u_{0}\|_{\mathcal{H}}+e^{-\delta t}\|u_{1}\|_{\mathcal{H}}&\text{if } t\in (0,1],\\
        e^{-\delta t}\left(\|u_{0}\|_{\mathcal{H}}+\|u_{1}\|_{\mathcal{H}}\right)&\text{if } t\in [1,\infty),
    \end{cases}\quad \text{when} \quad \beta\leq\theta,\nonumber\\
      \|\partial_{t}u\|_{\mathcal{H}}&\lesssim \begin{cases}
    t^{-1}\|u_{0}\|_{\mathcal{H}}+\|u_{1}\|_{\mathcal{H}}&\text{if } t\in (0,1],\\
        e^{-\delta t}\|u_{0}\|_{\mathcal{H}}+\|u_{1}\|_{\mathcal{H}}&\text{if } t\in [1,\infty),
    \end{cases}\quad \text{and}\nonumber\\
    \|\partial_{t}^{k}u(t)\|_{\mathcal{H}}&\lesssim
\begin{cases}
 t^{-k}\|u_{0}\|_{\mathcal{H}}+t^{-(k-1)}\|u_{1}\|_{\mathcal{H}}&\text{if } t\in (0,1],\\
   e^{-\delta t}\left(\|u_{0}\|_{\mathcal{H}}+\|u_{1}\|_{\mathcal{H}}\right)&\text{if } t\in [1,\infty),
\end{cases}\quad \text{when} \quad k\geq 2,\nonumber
\end{align}    
for some $\delta>0$ and for all $t>0$.
\end{theorem}
At first look, the decay estimates in the critical scenario  may appear to follow immediately from the study of the effective damping scenario. This is not the case, though, and the critical regime needs to be handled independently.
Interestingly, the calculations required for the critical situation are far easier than those needed for the effective and non-effective damping cases, even though an independent handling is still necessary. The next theorem provides the decay estimates for the case of non-effective damping:
\begin{theorem}[\textbf{Non-effective damping}]\label{thm:case3}
Let $2\theta\in (\sigma,2\sigma)$, $\beta>0$, and $k\in\mathbb{N}$. Let $u$ be the solution of the following homogeneous Cauchy problem \eqref{main:linear:Hilbert},
then the solution $u$ satisfies the following estimates
\begin{align}
\|u(t)\|_{\mathcal{H}}&\lesssim\|u_{0}\|_{\mathcal{H}}+t\|u_{1}\|_{\mathcal{H}},\nonumber\\
\|\mathcal{L}^{\beta}u\|_{\mathcal{H}}&\lesssim \begin{cases}
        t^{-\frac{\beta}{\sigma-\theta}}\|u_{0}\|_{\mathcal{H}}+t^{-\frac{2\beta-\sigma}{2\sigma-2\theta}}\|u_{1}\|_{\mathcal{H}}&\text{if } t\in (0,1],\\
        e^{-\delta t}\left(\|u_{0}\|_{\mathcal{H}}+\|u_{1}\|_{\mathcal{H}}\right)&\text{if } t\in [1,\infty),
    \end{cases}\quad \text{when} \quad 2\beta>\sigma,\nonumber\\
    \|\mathcal{L}^{\beta}u\|_{\mathcal{H}}&\lesssim \begin{cases}
        t^{-\frac{\beta}{\sigma-\theta}}\|u_{0}\|_{\mathcal{H}}+e^{-\delta t}\|u_{1}\|_{\mathcal{H}}&\text{if } t\in (0,1],\\
        e^{-\delta t}\left(\|u_{0}\|_{\mathcal{H}}+\|u_{1}\|_{\mathcal{H}}\right)&\text{if } t\in [1,\infty),
    \end{cases}\quad \text{when} \quad 2\beta\leq\sigma,\nonumber\\
      \|\partial_{t}u\|_{\mathcal{H}}&\lesssim \begin{cases}
        t^{-\frac{\theta}{\sigma-\theta}}\|u_{0}\|_{\mathcal{H}}+\|u_{1}\|_{\mathcal{H}}&\text{if } t\in (0,1],\\
        e^{-\delta t}\|u_{0}\|_{\mathcal{H}}+\|u_{1}\|_{\mathcal{H}}&\text{if } t\in [1,\infty),
    \end{cases}\quad \text{and}\nonumber\\
    \|\partial_{t}^{k}u(t)\|_{\mathcal{H}}&\lesssim
\begin{cases}
 t^{-\frac{k\theta}{\sigma-\theta}}\|u_{0}\|_{\mathcal{H}}+t^{-\frac{(k-1)\theta}{\sigma-\theta}}\|u_{1}\|_{\mathcal{H}}&\text{if } t\in (0,1],\\
   e^{-\delta t}\left(\|u_{0}\|_{\mathcal{H}}+\|u_{1}\|_{\mathcal{H}}\right)&\text{if } t\in [1,\infty),
\end{cases}\quad \text{when} \quad k\geq 2,\nonumber
\end{align}    
for some $\delta>0$ and for all $t>0$.
\end{theorem}

\begin{remark}
    Here are some crucial observations from the decay estimates:
\begin{enumerate}
    \item In contrast to the classical wave equation in Euclidean case, we obtain exponential decay for higher-order time and spatial derivatives of the solution in the damped cases, at least for large timeframes $t\geq 1$; see Theorem \ref{thm:case1}, \ref{thm:case2}, and \ref{thm:case3}. This is because of the discrete spectrum of the positive operator $\mathcal{L}$. However, such estimates are not possible in the case of operator having continuous spectrum; see \cite{Bui,Iketha:struc}.
    \item However, close to the initial time $t=0$, the behavior is significantly distinct: the solution frequently exhibits a singularity, and the decay—if it occurs—is at most polynomial on the interval $t\in(0,1]$.  
   \item  The absence of decay or the presence of polynomial-time decay in the interval $t\in(0,1]$ is generally inadequate to demonstrate global well-posedness. The estimates in Theorem \ref{thm:expdecay} are best suited to prove global well-posedness, while the estimates in Theorem \ref{thm:case0old} can be used to obtain at least local well-posedness.
\end{enumerate}
\end{remark}
\begin{remark}
    The proofs for the effective, critical, and non-effective damping cases in Theorem \ref{thm:case0old} will be provided as remarks, following the proofs of Theorems \ref{thm:case1}, \ref{thm:case2}, and \ref{thm:case3}, respectively. The case of viscoelastic-type damping in Theorem \ref{thm:case0old} will be provided independently.
\end{remark}


 This concludes the outline of decay estimates; hereafter in Section \ref{section-7}, we will provide an application of the decay estimates.
\section{Preliminaries}\label{sec:prelim}
The key components of Fourier analysis in relation to our indicated space and the  self-adjoint, positive operator $\mathcal{L}$ will be reviewed in this section, see \cite{DMT17,Ruz16,ruzhansky2017nonharmonic,MR&JPVR} for a thorough exposition.
More precisely, this section defines the $\mathcal{L}$-Fourier transform,  $\mathcal{L}$-convolution and the Sobolev spaces associated with the operator $\mathcal{L}$ and its properties.  At the end, we will conclude this section by recalling some important inequalities.

For the  self-adjoint, positive operator $\mathcal{L}$,
		the space $H_{\mathcal{L}}^{\infty}:=$ $\operatorname{Dom}\left(\mathcal{L}^{\infty}\right)$ is referred to as  the space of test functions   given by 
		$$
\operatorname{Dom}\left(\mathcal{L}^{\infty}\right):=\bigcap_{k=1}^{\infty} \operatorname{Dom}\left(\mathcal{L}^{k}\right),
		$$
		where 	$\operatorname{Dom}\left(\mathcal{L}^{k}\right):=\left\{f \in \mathcal{H}: \mathcal{L}^{j} f \in \mathcal{H}, j=0,1,2, \ldots, k-1\right\}
		$
        represents the domain of $\mathcal{L}^{k}$ and the
	 Fréchet topology of $H_{\mathcal{L}}^{\infty}$ is given by the family of norms
	$$
\|\varphi\|_{H_{\mathcal{L}}^{k}}:=\max _{j \leq k}\left\|\mathcal{L}^{j} \varphi\right\|_{\mathcal{H}},~~ k \in \mathbb{N}_{0}, ~\varphi \in H_{\mathcal{L}}^{\infty}.
	$$
    We define the space of $\mathcal{L}$-distributions as the space of linear continuous functionals on $H_{\mathcal{L}}^{\infty}$ i.e., 
	$\mathcal{D}^{\prime}(\mathcal{L}) :=\mathcal{L}(H_{\mathcal{L}}^{\infty}, \mathbb{C} )$.
    
Let $\{\lambda_{\xi}:\xi\in \mathcal{I}\}$ denote the set of eigenvalues of the  self-adjoint, positive operator $\mathcal{L}$, with $u_{\xi}\in\mathcal{H}$ denoting the associated eigenfunctions. The space of rapidly decreasing function on $\mathcal{I}$ denoted by $\mathcal{S}(\mathcal{I})$ is defined as follows:
	  $\phi \in \mathcal{S}(\mathcal{I})$ if for every $N\in\mathbb{N}_{0},$ there exists a constant $C_{\phi, N}$ satisfying
	$$
	|\phi(\xi)| \leq C_{\phi,N}\langle\xi\rangle^{-N}\quad \text{for all }\xi \in \mathcal{I},
	$$
	 where $\langle\xi\rangle=(1+\lambda_{\xi})^{\frac{1}{2}}$. The topology on $\mathcal{S}(\mathcal{I})$ is given by the seminorms $p_{k}, k \in \mathbb{N}_{0}$ and is defined by 
	$$
	p_{k}(\phi):=\sup _{\xi \in \mathcal{I}}\langle\xi\rangle^{k}|\phi(\xi)|.
	$$
We can now introduce the notion of $\mathcal{L}$-Fourier transform.
	\begin{defn}
	    [Fourier transform]\label{LFT}
		The $\mathcal{L}$-Fourier transform
		$\mathcal{F}_{\mathcal{L}} : H_{\mathcal{L}}^{\infty} \rightarrow \mathcal{S}(\mathcal{I})
		$
		is given by
		$$
		\widehat{f}(\xi):=\left(\mathcal{F}_{\mathcal{L}} f\right)(\xi)=\langle f,u_{\xi}\rangle.
		$$
		 Consequently, the inverse Fourier transform $\mathcal{F}_{\mathcal{L}}^{-1}: \mathcal{S}(\mathcal{I}) \rightarrow H_{\mathcal{L}}^{\infty}$ is given by
		$$
\left(\mathcal{F}_{\mathcal{L}}^{-1} h \right)(x)=\sum_{\xi \in \mathcal{I}} h(\xi) u_{\xi}(x),~ h \in \mathcal{S}(\mathcal{I})
		$$
and the Fourier inversion formula takes the form
		$$
		f(x)=\sum_{\xi \in \mathcal{I}} \widehat{f}(\xi) u_\xi(x),~~ f \in H_{\mathcal{L}}^{\infty}.
		$$
	\end{defn}
	 The Plancherel's formula is given by
        \begin{equation}\label{planchform}
  \|f\|_{\mathcal{H}}=\left(\sum_{\xi\in\mathcal{I}}|\widehat{f}(\xi)|^2\right)^{1/2}.
        \end{equation}
		The  $\mathcal{L}$-convolution of  $f_1, f_2 \in H_{\mathcal{L}}^{\infty}$  is defined by
        \begin{equation}\label{convolution}
            	\left(f_1*_{\mathcal{L}} f_2\right)(x):=\sum_{\xi \in \mathcal{I}} \widehat{f_1}(\xi) \widehat{f_2}(\xi)  u_\xi(x).
        \end{equation}
        We now recall Sobolev spaces generated by the operator $\mathcal{L}$.
		For $f \in \mathcal{D}^{\prime}(\mathcal{L})$ and $s>0,$ we say that $f \in \mathcal{H}_{\mathcal{L}}^{s}$ iff $\mathcal{L}^{s/2}f\in\mathcal{H}$ and the associated Sobolev norm is defined as
		$$
\|f\|_{\mathcal{H}_{\mathcal{L}}^{s}}:=\|\mathcal{L}^{s/2}f\|_{\mathcal{H}}=\left(\sum_{\xi \in \mathcal{I}}\lambda_{\xi}^{s} |\widehat{f}(\xi)|^2\right)^{1/2}.
		$$
    Moreover, for every $s >0,$ the Sobolev space $\mathcal{H}_{ \mathcal{L}}^{s}$ is a Hilbert space with the inner product given by 
	$$
	\langle f, g\rangle_{\mathcal{H}_{ \mathcal{L}}^{s}}:=\sum_{\xi \in \mathcal{I}}\lambda_{\xi}^{s} \widehat{f}(\xi) \overline{\widehat{g}(\xi)}.
	$$
   The following well-known inequalities will be our crucial ingredients to develop the linear estimates for the Cauchy problem \eqref{main:linear:Hilbert}.
    \begin{enumerate}
        \item\label{rmk1} For $x\in[0,\frac{1}{4}]$, the following inequalities hold:
        \begin{equation}
   -4x\leq -1+\sqrt{1-4x}\leq -2x\quad \text{and}\quad
   -2\leq -1-\sqrt{1-4x}\leq -1.\label{INEQREL1}
\end{equation}
\item For any $\gamma,\delta>0$, the function $f(t)=t^{\gamma}e^{-\delta t}$ and $f(t)=(1+t)^{\gamma}e^{-\delta t}$ is bounded on $[0,\infty)$. Furthermore,
for any $\gamma>0$ and $\beta>\delta>0$, we have 
    \begin{equation}\label{rmk2}
        t^{\gamma}e^{-\beta t}, (1+t)^{\gamma}e^{-\delta t} \leq C(\gamma,\delta)e^{-(\beta-\delta)t} \quad \text{for all } t> 0,
    \end{equation}
    for some positive constant $C(\gamma,\delta)>0$.
    \item Since $\mathcal{L}$ has a discrete spectrum and has no cluster point, therefore,  for given any two fixed numbers $\gamma,a>0$  there exist  constants
$\delta_{a,\gamma}>0$ and $\delta^{a,\gamma}>0$ such that
    \begin{equation}\label{rmk3}
       \sup_{\xi\in\mathcal{I}}\{\lambda_{\xi}^{\gamma }:0<\lambda_{\xi}^{\gamma}<a \}<\delta^{a,\gamma} <a<\delta_{a,\gamma}<\inf_{\xi\in\mathcal{I}}\{\lambda_{\xi}^{\gamma }:\lambda_{\xi}^{\gamma}>a \}.
    \end{equation}
Since ensuring that $\delta$ remains independent of $\xi\in\mathcal{I}$ is adequate for our calculations, we will suppress the dependence of $\delta$ on the parameters $a$ and $\gamma$.
    \end{enumerate}
\section{Preparatory Estimates}\label{section-4}
This section is devoted for some primary computation that will be required to derive the decay estimates for the homogeneous Cauchy problem \eqref{main:linear:Hilbert}.

 Taking the $\mathcal{L}$-Fourier transform \eqref{LFT} to the  Cauchy problem \eqref{main:linear:Hilbert}, we obtain
 \begin{align}\label{hom:eq1}
	\begin{cases}
		\partial^{2}_{t}\widehat{u}(t,\xi)+\lambda_{\xi}^{\theta}\widehat{u}_{t}(t,\xi)+\lambda_{\xi}^{\sigma} \widehat{u}(t,\xi)=0,& \xi\in\mathcal{I},~t>0,\\ \widehat{u}(0,\xi)=\widehat{u}_0(\xi), \quad \partial_{t}\widehat{u}(0,\xi)=\widehat{u}_{1}(\xi), &\xi\in\mathcal{I},
	\end{cases} 
\end{align}
where $\{\lambda_{\xi}\}_{\xi\in\mathcal{I}}$  are  discrete eigenvalues of the operator $\mathcal{L}$.
Consequently,
the characteristic equation of the Cauchy problem \eqref{hom:eq1} is given by
\[\tau ^2+ \lambda_{\xi}^{\theta} \tau  +\lambda_{\xi}^{\sigma} =0,\]
and the characteristic roots are given by
\begin{equation*}
    \tau_{\pm}=\begin{cases}
	-\frac{1}{2}\lambda_{\xi}^{\theta}\pm \frac{1}{2}\sqrt{ \lambda_{\xi}^{2\theta}-4\lambda_{\xi}^{\sigma}}& \text{if }  4\lambda_{\xi}^{\sigma}<\lambda_{\xi}^{2\theta},\\
		-\frac{1}{2}\lambda_{\xi}^{\theta}&\text{if }  4\lambda_{\xi}^{\sigma}=\lambda_{\xi}^{2\theta},\\
			-\frac{1}{2}\lambda_{\xi}^{\theta}\pm \frac{i}{2}\sqrt{4\lambda_{\xi}^{\sigma}-\lambda_{\xi}^{2\theta}}&\text{if } 4\lambda_{\xi}^{\sigma}>\lambda_{\xi}^{2\theta}.
	\end{cases}
\end{equation*}
Depending on the range of $\theta$ and $\sigma$, the characteristic roots take the form
\begin{align}
\text{undamped i.e., }\theta=0 : ~~\tau_{\pm}&=\begin{cases}
	-\frac{1}{2}\pm \sqrt{\frac{1}{4}-\lambda_{\xi}^{\sigma}}& \text{if }  \lambda_{\xi}^{\sigma}<\frac{1}{4},\\
		-\frac{1}{2}&\text{if }  \lambda_{\xi}^{\sigma}=\frac{1}{4},\\
			-\frac{1}{2}\pm i\sqrt{\lambda_{\xi}^{\sigma}-\frac{1}{4}}&\text{if } \lambda_{\xi}^{\sigma}>\frac{1}{4},
	\end{cases}\label{case0:root}\\
    \text{effective-damping i.e., } 2\theta\in(0,\sigma): ~~\tau_{\pm}&=\begin{cases}
	-\frac{1}{2}\lambda_{\xi}^{\theta}\pm \sqrt{\frac{1}{4} \lambda_{\xi}^{2\theta}-\lambda_{\xi}^{\sigma}}& \text{if }  \lambda_{\xi}^{\sigma-2\theta}<\frac{1}{4},\\
		-\frac{1}{2}\lambda_{\xi}^{\theta}&\text{if }  \lambda_{\xi}^{\sigma-2\theta}=\frac{1}{4},\\
			-\frac{1}{2}\lambda_{\xi}^{\theta}\pm i\sqrt{\lambda_{\xi}^{\sigma}-\frac{1}{4}\lambda_{\xi}^{2\theta}}&\text{if } \lambda_{\xi}^{\sigma-2\theta}>\frac{1}{4},
	\end{cases}\label{case1:root}\\
   \text{critical damping i.e., } 2\theta=\sigma :~~ \tau_{\pm}&= -\frac{1}{2}\lambda_{\xi}^{\theta}\pm i\frac{\sqrt{3}}{2}\lambda_{\xi}^{\theta},\label{case2:root}\\
   \text{non-effective-damping i.e., } 2\theta\in(\sigma,2\sigma):~~\tau_{\pm}&=\begin{cases}
		-\frac{1}{2}\lambda_{\xi}^{\theta}\pm i\sqrt{\lambda_{\xi}^{\sigma}-\frac{1}{4}\lambda_{\xi}^{2\theta}}&\text{if } \lambda_{\xi}^{2\theta-\sigma}<4,\\
		-\frac{1}{2}\lambda_{\xi}^{\theta}&\text{if }  \lambda_{\xi}^{2\theta-\sigma}=4,\\
            -\frac{1}{2}\lambda_{\xi}^{\theta}\pm \sqrt{\frac{1}{4} \lambda_{\xi}^{2\theta}-\lambda_{\xi}^{\sigma}}& \text{if }  \lambda_{\xi}^{2\theta-\sigma}>4,
	\end{cases}\label{case3:root}\\
    \text{viscoelastic type damping i.e., } \theta=\sigma:~~   \tau_{\pm}&=\begin{cases}
		-\frac{1}{2}\lambda_{\xi}^{\theta}\pm \frac{i}{2}\sqrt{4\lambda_{\xi}^{\theta}-\lambda_{\xi}^{2\theta}}&\text{if } \lambda_{\xi}^{\theta}<4,\\
		-\frac{1}{2}\lambda_{\xi}^{\theta}&\text{if }  \lambda_{\xi}^{\theta}=4,\\
            -\frac{1}{2}\lambda_{\xi}^{\theta}\pm \frac{1}{2}\sqrt{\lambda_{\xi}^{2\theta}-4\lambda_{\xi}^{\theta}}& \text{if }  \lambda_{\xi}^{\theta}>4.
	\end{cases}\label{case3:root new}
\end{align}
 The solution of the Cauchy problem \eqref{main:linear:Hilbert} can be written as
\begin{equation}
u(t)=E_{0}(t)*_{\mathcal{L}}u_{0}+E_{1}(t)*_{\mathcal{L}}u_{1},\nonumber
\end{equation}
where $*_{\mathcal{L}}$ is the convolution defined in \eqref{convolution}, equivalently
\begin{equation}\label{usoln:fourier}
    \widehat{u}(t,\xi)=\widehat{E}_{0}(t,\xi)\widehat{u}_{0}(\xi)+\widehat{E}_{1}(t,\xi)\widehat{u}_{1}(\xi),
\end{equation}
where 
\begin{equation}\label{e0e1formula}
    \widehat{E}_{0}(t,\xi)=\begin{cases}
        \frac{\tau_{+}e^{\tau_{-}t}-\tau_{-}e^{\tau_{+}t}}{\tau_{+}-\tau_{-}}&\text{if } \tau_{+}\neq \tau_{-},\\
        \left(1-\tau_{+} t\right)e^{\tau_{+} t}&\text{if }\tau_{+}=\tau_{-},
    \end{cases}~~\text{and}~~\widehat{E}_{1}(t,\xi)=\begin{cases}
        \frac{e^{\tau_{+}t}-e^{\tau_{-}t}}{\tau_{+}-\tau_{-}}&\text{if } \tau_{+}\neq \tau_{-},\\
         te^{\tau_{+} t}&\text{if }\tau_{+}=\tau_{-}.
    \end{cases}
\end{equation}
Consequently, for $k\in\mathbb{N}$ we have
\begin{equation}\label{timedersoln:fourier}
    \partial_{t}^{k}\widehat{u}(t,\xi)=\partial_{t}^{k}\widehat{E}_{0}(t,\xi)\widehat{u}_{0}(\xi)+\partial_{t}^{k}\widehat{E}_{1}(t,\xi)\widehat{u}_{1}(\xi),
\end{equation}
where
\begin{align}\label{timederive0}
    \partial^{k}_{t}\widehat{E}_{0}(t,\xi)
     &=-\tau_{+}\tau_{-}\partial^{k-1}_{t}\widehat{E}_{1}(t,\xi)=-\lambda_{\xi}^{\sigma}\partial^{k-1}_{t}\widehat{E}_{1}(t,\xi),\quad\text{and}\\ \partial^{k}_{t}\widehat{E}_{1}(t,\xi)&=\begin{cases}
          \frac{(\tau_{+})^{k}e^{\tau_{+}t}-(\tau_{-})^{k}e^{\tau_{-}t}}{\tau_{+}-\tau_{-}}&\text{if } \tau_{+}\neq \tau_{-},\\
         (\tau_{+})^{k-1}(k+t\tau_{+})e^{\tau_{+} t}&\text{if }\tau_{+}=\tau_{-}.
     \end{cases}\label{timederive1}
\end{align}

\section{Proof of main results}\label{section-5}
This section presents the proof of our main results, which is reliant upon the estimates of $\widehat{E}_{0}(t,\xi)$ and $\widehat{E}_{1}(t,\xi)$, as well as their spatial and temporal derivatives.  The proofs of Theorem \ref{thm:case0old} and \ref{thm:expdecay} will be presented gradually, accompanied by a set of remarks subsequent to the derivation of decay estimates in each respective case. 
 \begin{proof}[Proof of Theorem \ref{thm:case0new}]
Let's begin to estimate $\widehat{E}_{0}(t,\xi)$ and $\widehat{E}_{1}(t,\xi)$ on $\mathcal{I}$. To simplify the presentation, we introduce the following partition of $\mathcal{I}$ in this case:
    \begin{enumerate}
    \item $\mathcal{R}_{1}=\{\xi\in \mathcal{I}: \lambda_{\xi}^{\sigma}=0\},$
    \item  $\mathcal{R}_{2}=\left\{\xi\in \mathcal{I}: 0<\lambda_{\xi}^{\sigma}<\frac{1}{4}\right\}$,
    \item $\mathcal{R}_{3}=\left\{\xi\in \mathcal{I}: \lambda_{\xi}^{\sigma}=\frac{1}{4}\right\}$, and
    \item $ \mathcal{R}_{4}=\left\{\xi\in \mathcal{I}: \lambda_{\xi}^{\sigma}>\frac{1}{4}\right\}$.
\end{enumerate}
Combining the relation \eqref{case0:root} with representation \eqref{e0e1formula}, we get
\begin{equation}\label{case0:e0def}
    \widehat{E}_{0}(t,\xi)=\begin{cases}1 &\text{if } \xi\in\mathcal{R}_{1},\\
    \left[\cosh \left(\sqrt{\frac{1}{4}-\lambda_{\xi}^{\sigma}} t\right)+\frac{\sinh \left(\sqrt{\frac{1}{4}-\lambda_{\xi}^{\sigma}} t\right)}{2\sqrt{\frac{1}{4}-\lambda_{\xi}^{\sigma}}} \right]e^{-\frac{1}{2}t}& \text {if } \xi\in\mathcal{R}_{2},\\
   \left[1+\frac{1}{2}t\right]e^{-\frac{1}{2}t}& \text {if } \xi\in\mathcal{R}_{3}, \\\left[ \cos \left(\sqrt{\lambda_{\xi}^{\sigma}-\frac{1}{4}} t\right)+\frac{\sin\left(\sqrt{\lambda_{\xi}^{\sigma}-\frac{1}{4}} t\right)}{2\sqrt{\lambda_{\xi}^{\sigma}-\frac{1}{4}}}\right]e^{-\frac{1}{2}t}  & \text {if }  \xi\in\mathcal{R}_{4},\end{cases}
\end{equation}
and
\begin{equation}\label{case0:e1def}
    \widehat{E}_{1}(t,\xi)=\begin{cases}1-e^{-t} &\text{if } \xi\in\mathcal{R}_{1},\\
    \frac{\sinh \left(\sqrt{\frac{1}{4}-\lambda_{\xi}^{\sigma}} t\right)}{\sqrt{\frac{1}{4}-\lambda_{\xi}^{\sigma}}}e^{-\frac{1}{2}t} & \text {if } \xi\in\mathcal{R}_{2},\\ te^{-\frac{1}{2}t} & \text {if } \xi\in\mathcal{R}_{3},\\ \frac{\sin \left(\sqrt{\lambda_{\xi}^{\sigma}-\frac{1}{4}} t\right)}{\sqrt{\lambda_{\xi}^{\sigma}-\frac{1}{4}}}e^{-\frac{1}{2}t} & \text {if } \xi\in\mathcal{R}_{4}.\end{cases}
\end{equation}
Now we will estimate $\widehat{E}_{0}(t,\xi)$ and $\widehat{E}_{1}(t,\xi)$ on $\mathcal{R}_{j}$ for $j=1,2,3,4$.

\medskip

\noindent\textbf{Estimate on $\mathcal{R}_{1}$:} From \eqref{case0:e0def} and \eqref{case0:e1def},  we have
\begin{equation*}
    |\widehat{E}_{0}(t,\xi)|= 1\quad\text{and}\quad     |\widehat{E}_{1}(t,\xi)|\lesssim 1,\nonumber
\end{equation*}
and  hence the representation \eqref{usoln:fourier} gives
\begin{equation}\label{c1R11:usol}
    |\widehat{u}(t,\xi)|\lesssim |\widehat{u}_{0}(\xi)| +|\widehat{u}_{1}(\xi)|.
\end{equation}
Furthermore, for any $\beta>0$ we have
\begin{equation}\label{c1R11:spacedersol}
    |\lambda_{\xi}^{\beta}\widehat{u}(t,\xi)|=0.
\end{equation}
Using the relations \eqref{timederive0} and \eqref{timederive1} for the equations \eqref{case0:e0def} and \eqref{case0:e1def}, respectively, we have
\begin{align}
   \partial^{k}_{t}\widehat{E}_{0}(t,\xi)=0\quad&\text{and}\quad \partial^{k}_{t}\widehat{E}_{1}(t,\xi)= (-1)^{k-1}e^{-t}, \quad \text{for } k\geq 1, \nonumber
\end{align}
and hence the representation \eqref{timedersoln:fourier} along with the above relations gives
\begin{equation}\label{c1R11:timedersoln}
       |\partial^{k}_{t}\widehat{u}(t,\xi)|\lesssim e^{-t}|\widehat{u}_{1}(\xi)|\quad\text{for } k\geq 1.
\end{equation}
\noindent\textbf{Estimate on $\mathcal{R}_{2}$:} From \eqref{case0:e0def} and \eqref{case0:e1def}, we have
\begin{equation}
    |\widehat{E}_{0}(t,\xi)|\lesssim  \left[1+\frac{1}{2\sqrt{\frac{1}{4}-\lambda_{\xi}^{\sigma}}}\right]e^{\left(-\frac{1}{2}+\sqrt{\frac{1}{4}-\lambda_{\xi}^{\sigma}}\right)t}\lesssim e^{-\lambda_{\xi}^{\sigma}t},\nonumber
\end{equation}
and
\begin{equation}
      |\widehat{E}_{1}(t,\xi)|\lesssim  \frac{1}{\sqrt{\frac{1}{4}-\lambda_{\xi}^{\sigma}}}e^{\left(-\frac{1}{2}+\sqrt{\frac{1}{4}-\lambda_{\xi}^{\sigma}}\right)t}\lesssim e^{-\lambda_{\xi}^{\sigma} t},\nonumber
\end{equation}
where we have utilized the relation \eqref{INEQREL1}
to estimate the exponential factor and the fact that $\{\lambda_{\xi}\}_{\xi\in\mathcal{I}}$ is discrete, which allows us to bound  $\frac{1}{\sqrt{\frac{1}{4}-\lambda_{\xi}^{\sigma}}}$ uniformly. Consequently, the representation \eqref{usoln:fourier} along with the above estimates gives
\begin{equation}\label{c1R12:usoln}
    |\widehat{u}(t,\xi)|\lesssim e^{-\lambda_{\xi}^{\sigma}t}\left(|\widehat{u}_{0}(\xi)| +|\widehat{u}_{1}(\xi)|\right)\lesssim e^{-\delta t}\left(|\widehat{u}_{0}(\xi)| +|\widehat{u}_{1}(\xi)|\right),
\end{equation}
where the constant $\delta$ is obtained by using the relation \eqref{rmk3} which is due to discrete spectrum.
Furthermore, for any $\beta>0$ we have
\begin{equation}\label{c1R12:spaceder}
\lambda_{\xi}^{\beta}|\widehat{u}(t,\xi)|\lesssim \lambda_{\xi}^{\beta}e^{-\lambda_{\xi}^{\sigma}t}\left(|\widehat{u}_{0}(t,\xi)|+|\widehat{u}_{1}(t,\xi)|\right)\lesssim e^{-\delta t}\left(|\widehat{u}_{0}(t,\xi)|+|\widehat{u}_{1}(t,\xi)|\right),
\end{equation}
since $\lambda_{\xi}^{\beta}$ is uniformly bounded on $\mathcal{R}_{2}$.
Recalling the relation \eqref{timederive1} for $\widehat{E}_{1}(t,\xi)$, we have
\begin{equation*}
    \partial_{t}^{k}\widehat{E}_{1}(t,\xi)=\frac{\left(-\frac{1}{2}+\sqrt{\frac{1}{4}-\lambda_{\xi}^{\sigma}}\right)^{k}}{2\sqrt{\frac{1}{4}-\lambda_{\xi}^{\sigma}}}e^{\left(-\frac{1}{2}+\sqrt{\frac{1}{4}-\lambda_{\xi}^{\sigma}}\right)t}-\frac{\left(-\frac{1}{2}-\sqrt{\frac{1}{4}-\lambda_{\xi}^{\sigma}}\right)^{k}}{2\sqrt{\frac{1}{4}-\lambda_{\xi}^{\sigma}}}e^{\left(-\frac{1}{2}-\sqrt{\frac{1}{4}-\lambda_{\xi}^{\sigma}}\right)t}.
\end{equation*}
This gives
\begin{equation*}
    |\partial_{t}^{k}\widehat{E}_{1}(t,\xi)|\lesssim e^{-\lambda_{\xi}^{\sigma}t}+e^{-\frac{1}{2}t}\lesssim  e^{-\min\{\frac{1}{2},\delta \}t}\lesssim e^{-\delta t},\nonumber
\end{equation*}
for some modified $\delta>0$, where we have used the facts that $\mathcal{L}$ has a discrete spectrum with no cluster point, and $\frac{\left(-\frac{1}{2}+\sqrt{\frac{1}{4}-\lambda_{\xi}^{\sigma}}\right)^{k}}{2\sqrt{\frac{1}{4}-\lambda_{\xi}^{\sigma}}}$ and $\frac{\left(-\frac{1}{2}-\sqrt{\frac{1}{4}-\lambda_{\xi}^{\sigma}}\right)^{k}}{2\sqrt{\frac{1}{4}-\lambda_{\xi}^{\sigma}}}$ are uniformly bounded on $\mathcal{R}_{2}$.
Consequently, using the relation \eqref{timederive0}, we obtain
\begin{equation*}
    |\partial_{t}^{k}\widehat{E}_{0}(t,\xi)|\lesssim \lambda_{\xi}^{\sigma}e^{-\delta t}\lesssim e^{-\delta t},
\end{equation*}
since $\lambda_{\xi}^{\sigma}$ is uniformly bounded on $\mathcal{R}_{2}$.
Combining the above estimate with the representation \eqref{timedersoln:fourier}, we get
\begin{equation}\label{c1R12:timedersoln}
    |\partial_{t}^{k}\widehat{u}(t,\xi)|\lesssim e^{-\delta t} \left(|\widehat{u}_{0}(t,\xi)|+|\widehat{u}_{1}(t,\xi)|\right).
\end{equation}
\noindent\textbf{Estimate on $\mathcal{R}_{3}$:} From \eqref{case0:e0def} and \eqref{case0:e1def}, we have
\begin{equation}
     |\widehat{E}_{0}(t,\xi)|\lesssim (1+t)e^{-\frac{1}{2}t}\quad\text{and}\quad     |\widehat{E}_{1}(t,\xi)|\lesssim te^{-\frac{1}{2}t},\nonumber
\end{equation}
this gives
\begin{equation}\label{c1R13:usoln}
    |\widehat{u}(t,\xi)|\lesssim (1+t) e^{-\frac{1}{2}t}|\widehat{u}_{0}(\xi)| +t e^{-\frac{1}{2} t}|\widehat{u}_{1}(\xi)|.
\end{equation}
Consequently, for any $\beta>0$, 
\begin{equation}\label{c1R13:spacedersoln}
    \lambda_{\xi}^{\beta}|\widehat{u}(t,\xi)|\lesssim (1+t)e^{-\frac{1}{2} t}|\widehat{u}_{0}(\xi)| +te^{-\frac{1}{2} t}|\widehat{u}_{1}(\xi)|,
\end{equation}
Furthermore, using the relation \eqref{timederive1} for $\widehat{E}_{1}^{1}(t,\xi)$, we have
\begin{equation}
    |\partial^{k}_{t}\widehat{E}_{1}(t,\xi)|=\left|\left(-\frac{1}{2}\right)^{k}\left(k-\frac{1}{2}t\right)e^{-\frac{1}{2}t}\right|\lesssim (1+t)e^{-\frac{1}{2} t},\nonumber
\end{equation}
and consequently, using the relation
\eqref{timederive0}, we obtain
\begin{equation}
    |\partial^{k}_{t}\widehat{E}_{0}(t,\xi)|\lesssim (1+t)e^{-\frac{1}{2} t}.\nonumber
\end{equation}
Combining the above estimate with the representation \eqref{timedersoln:fourier}, we have
\begin{equation}\label{c1R13:timedersoln}
    |\partial_{t}^{k}\widehat{u}(t,\xi)|\lesssim (1+t)e^{-\frac{1}{2} t}\left(|\widehat{u}_{0}(\xi)| +|\widehat{u}_{1}(\xi)|\right),\quad \text{for } k\geq 1.
\end{equation}
\noindent\textbf{Estimate on $\mathcal{R}_{4}$:} From \eqref{case0:e0def} and \eqref{case0:e1def}, we get
\begin{align*}
     |\widehat{E}_{0}(t,\xi)|,~ |\widehat{E}_{1}(t,\xi)|\lesssim e^{-\frac{1}{2}t},
     \end{align*}
where we have utilized \eqref{rmk3} to bound the denominator term in \eqref{case0:e0def} and \eqref{case0:e1def}. Consequently, we get
\begin{equation}\label{c1R14:usoln}
    |\widehat{u}(t,\xi)|\lesssim e^{-\frac{1}{2}t}\left(|\widehat{u}_{0}(\xi)| +|\widehat{u}_{1}(\xi)|\right).
\end{equation}
Further, for any $\beta>0$, the term $|\lambda_{\xi}^{\beta}\widehat{E}_{0}(t,\xi)|$ and $|\lambda_{\xi}^{\beta}\widehat{E}_{1}(t,\xi)|$ can be estimated as
\begin{equation*}
     |\lambda_{\xi}^{\beta}\widehat{E}_{0}(t,\xi)|\lesssim \lambda_{\xi}^{\beta}e^{-\frac{1}{2}t}\quad \text{and}\quad |\lambda_{\xi}^{\beta}\widehat{E}_{1}(t,\xi)|\lesssim  \frac{\lambda_{\xi}^{\beta}}{\sqrt{\lambda_{\xi}^{\sigma}-\frac{1}{4}}}e^{-\frac{1}{2}t}\lesssim \lambda_{\xi}^{\beta-\frac{\sigma}{2}}e^{-\frac{1}{2}t},
\end{equation*}
since $\frac{1}{\sqrt{1-\frac{1}{4}\lambda_{\xi}^{-\sigma}}}$ is uniformly bounded due to discrete spectrum.
Consequently, using the above estimates with \eqref{usoln:fourier} we obtain
\begin{equation}\label{c1R14:spacedersoln}
|\lambda_{\xi}^{\beta}\widehat{u}(t,\xi)|\lesssim e^{-\frac{1}{2}t}
\left(\lambda_{\xi}^{\beta}|\widehat{u}_{0}(\xi)|+\lambda_{\xi}^{\beta-\frac{\sigma}{2}}|\widehat{u}_{1}(\xi)|\right).
\end{equation}
Recalling the relation \eqref{timederive1} for $\widehat{E}_{1}(t,\xi)$, we have
\small{\begin{equation*}\label{eider}
    \partial_{t}^{k}\widehat{E}_{1}(t,\xi)=\frac{\left(-\frac{1}{2}+i\sqrt{\lambda_{\xi}^{\sigma}-\frac{1}{4}}\right)^{k}}{2i\sqrt{\lambda_{\xi}^{\sigma}-\frac{1}{4}}}e^{\left(-\frac{1}{2}+i\sqrt{\lambda_{\xi}^{\sigma}-\frac{1}{4}}\right)t}-\frac{\left(-\frac{1}{2}-i\sqrt{\lambda_{\xi}^{\sigma}-\frac{1}{4}}\right)^{k}}{2i\sqrt{\lambda_{\xi}^{\sigma}-\frac{1}{4}}}e^{\left(-\frac{1}{2}-i\sqrt{\lambda_{\xi}^{\sigma}-\frac{1}{4}}\right)t},
\end{equation*}}
and using the boundedness of $\frac{1}{\sqrt{\lambda_{\xi}^{\sigma}-\frac{1}{4}}}$, we obtain
\small{\begin{align*}
    |\partial_{t}^{k}\widehat{E}_{1}(t,\xi)|\lesssim\left(\left|\left(-\frac{1}{2}+i\sqrt{\lambda_{\xi}^{\sigma}-\frac{1}{4}}\right)^{k-1}\right|+\left|\left(-\frac{1}{2}-i\sqrt{\lambda_{\xi}^{\sigma}-\frac{1}{4}}\right)^{k-1}\right|\right)e^{-\frac{1}{2}t}\lesssim \lambda_{\xi}^{\frac{k-1}{2}\sigma}e^{-\frac{1}{2}t}.
\end{align*}}
where we have used the fact that $|z^{k-1}|=|z|^{k-1}$. 
 Consequently using the relation \eqref{timederive0}, we get
\begin{equation}
    |\partial_{t}^{k}\widehat{E}_{0}(t,\xi)|\lesssim \lambda_{\xi}^{\sigma}|\partial_{t}^{k-1}\widehat{E}_{1}(t,\xi)|\lesssim \lambda_{\xi}^{\frac{k}{2}\sigma}e^{-\frac{1}{2} t}.\nonumber
\end{equation}
Hence we obtain
\begin{equation}\label{c1R14:timedersoln}
      |\partial_{t}^{k}\widehat{u}(t,\xi)|\lesssim e^{-\frac{1}{2}t}\left(\lambda_{\xi}^{\frac{k}{2}\sigma}|\widehat{u}_{0}(\xi)| +\lambda_{\xi}^{\frac{k-1}{2}\sigma}|\widehat{u}_{1}(\xi)|\right),\quad \text{for } k\geq 1.
\end{equation}
We are now in position to compute the estimate for $\|u(t)\|_{\mathcal{H}}, \|\mathcal{L}^{\beta}u(t)\|_{\mathcal{H}},$ and $\|\partial_{t}^{k}u(t)\|_{\mathcal{H}}$.

\medskip

\noindent\textbf{Estimate for $\|u(t)\|_{\mathcal{H}}$:} Recalling the Plancherel formula \eqref{planchform} along with the estimates \eqref{c1R11:usol}, \eqref{c1R12:usoln}, \eqref{c1R13:usoln}, and \eqref{c1R14:usoln}, we have
\begin{multline}\label{case1:uest new}
\|u(t)\|_{\mathcal{H}}^{2}=\sum_{\xi\in\mathcal{I}}|\widehat{u}(t,\xi)|^{2}=\sum_{j=1}^{4}\sum_{\xi\in\mathcal{R}_{j}}|\widehat{u}(t,\xi)|^{2}
\lesssim \sum_{\xi\in \mathcal{R}_{1}}\left(|\widehat{u}_{0}(\xi)|^{2} +|\widehat{u}_{1}(\xi)|^{2}\right)+\sum_{\xi\in \mathcal{R}_{2}}e^{-\delta t}\left(|\widehat{u}_{0}(\xi)|^{2} +|\widehat{u}_{1}(\xi)|^{2}\right)\\+ \sum_{\xi\in \mathcal{R}_{3}}(1+t)^{2}e^{-t}\left(|\widehat{u}_{0}(\xi)|^{2} +|\widehat{u}_{1}(\xi)|^{2}\right)
+\sum_{\xi\in \mathcal{R}_{4}}e^{-t}\left(|\widehat{u}_{0}(\xi)|^{2} +|\widehat{u}_{1}(\xi)|^{2}\right)
\lesssim\|u_{0}\|^{2}_{\mathcal{H}}+\|u_{1}\|^{2}_{\mathcal{H}}.
\end{multline}
\noindent\textbf{Estimate for $\|\mathcal{L}^{\beta}u(t)\|_{\mathcal{H}}$:} Similar to the above estimate using Plancherel formula along with the estimate
\eqref{c1R11:spacedersol}, \eqref{c1R12:spaceder}, \eqref{c1R13:spacedersoln}, and \eqref{c1R14:spacedersoln}, we obtain
\begin{multline}\label{case1:spacederuest}
\left\|\mathcal{L}^{\beta}u(t)\right\|_{\mathcal{H}}^{2}=\sum_{j=1}^{4}\sum_{\xi\in\mathcal{R}_{j}}\lambda_{\xi}^{2\beta}|\widehat{u}(t,\xi)|^{2}\lesssim \sum_{\xi\in \mathcal{R}_{2}}e^{-\delta t}\left(|\widehat{u}_{0}(\xi)|^{2} +|\widehat{u}_{1}(\xi)|^{2}\right)
+ \sum_{\xi\in \mathcal{R}_{3}}(1+t)^{2}e^{-t}\left(|\widehat{u}_{0}(\xi)|^{2} +|\widehat{u}_{1}(\xi)|^{2}\right)\\+\sum_{\xi\in \mathcal{R}_{4}}e^{-t}
\left(\lambda_{\xi}^{2\beta}|\widehat{u}_{0}(\xi)|^{2}+\lambda_{\xi}^{2\beta-\sigma}|\widehat{u}_{1}(\xi)|^{2}\right)\lesssim e^{-\delta t}\begin{cases}
\left(\|u_{0}\|^{2}_{\mathcal{H}^{2\beta}_{\mathcal{L}}}+\|u_{1}\|^{2}_{\mathcal{H}^{2\beta-\sigma}_{\mathcal{L}}}\right)&\text{if } 2\beta>\sigma,\\
\left(\|u_{0}\|^{2}_{\mathcal{H}^{2\beta}_{\mathcal{L}}}+\|u_{1}\|^{2}_{\mathcal{H}}\right)&\text{if } 2\beta\leq\sigma,
\end{cases}
\end{multline}
for some $\delta>0$, and we utilize the relation $(1+t)^{2}e^{-t}\lesssim e^{-\frac{1}{2}t}$ in the last inequality.

\medskip

\noindent\textbf{Estimate for $\|\partial_{t}^{k}u(t)\|_{\mathcal{H}}$:}
Combining the estimates \eqref{c1R11:timedersoln}, \eqref{c1R12:timedersoln}, \eqref{c1R13:timedersoln}, and \eqref{c1R14:timedersoln}, we get
\begin{multline}
\|\partial_{t}^{k}u(t)\|_{\mathcal{H}}^{2}\lesssim \sum_{\xi\in \mathcal{R}_{1}}e^{-2t}  |\widehat{u}_{1}(\xi)|^{2}+ \sum_{\xi\in \mathcal{R}_{2}}e^{-\delta t}  \left(|\widehat{u}_{0}(\xi)|^{2} +|\widehat{u}_{1}(\xi)|^{2}\right)
+ \sum_{\xi\in \mathcal{R}_{3}}(1+t)^{2}e^{-t}\left(|\widehat{u}_{0}(\xi)|^{2} +|\widehat{u}_{1}(\xi)|^{2}\right)\\+\sum_{\xi\in \mathcal{R}_{4}}e^{-t}\left(\lambda_{\xi}^{k\sigma}|\widehat{u}_{0}(\xi)|^{2} +\lambda_{\xi}^{(k-1)\sigma}|\widehat{u}_{1}(\xi)|^{2}\right)\lesssim
e^{-\delta t}\left(\|u_{0}\|^{2}_{\mathcal{H}^{k\sigma}_{\mathcal{L}}}+\|u_{1}\|^{2}_{\mathcal{H}^{(k-1)\sigma}_{\mathcal{L}}}\right).\nonumber
\end{multline}
This completes the proof.
\end{proof}
\begin{remark}
The following important observations can be drawn from the above proof:
\begin{itemize}
    \item The Remark \ref{remrk2:mthm} is now more understandable. Specifically, the estimate \eqref{case1:uest new} shows that the regions $\mathcal{R}_{2}$, $\mathcal{R}_{3}$, and $\mathcal{R}_{4}$ experience exponential decay. Nevertheless, the possibility of decay in the total energy norm $\|u(t)\|_{\mathcal{H}}$ is obstructed by the lack of any decay in $\mathcal{R}_{1}$.
    \item In $\mathcal{R}_{1}$, $\mathcal{R}_{2}$, and $\mathcal{R}_{3}$, the factor involving $\lambda_{\xi}$ can always be uniformly bounded; see \eqref{rmk3}. But in $\mathcal{R}_{4}$, such a uniform bound is not achievable, which highlights the minimal regularity required to control the associated norm.
\end{itemize}
\end{remark}
 Let us now begin to prove the decay estimates in the case of effective damping.
\begin{proof}[Proof of Theorem \ref{thm:case1}]
Let's begin to estimate $\widehat{E}_{0}(t,\xi)$ and $\widehat{E}_{1}(t,\xi)$ on $\mathcal{I}$. In this case, we introduce the following partition of $\mathcal{I}$:
    \begin{enumerate}
    \item $\mathcal{R}_{1}=\{\xi\in \mathcal{I}: \lambda_{\xi}^{\sigma-2\theta}=0\},$
    \item  $\mathcal{R}_{2}=\left\{\xi\in \mathcal{I}: 0<\lambda_{\xi}^{\sigma-2\theta}<\frac{1}{4}\right\}$,
    \item $\mathcal{R}_{3}=\left\{\xi\in \mathcal{I}: \lambda_{\xi}^{\sigma-2\theta}=\frac{1}{4}\right\}$, and
    \item $ \mathcal{R}_{4}=\left\{\xi\in \mathcal{I}: \lambda_{\xi}^{\sigma-2\theta}>\frac{1}{4}\right\}$.
\end{enumerate}
Combining the relation \eqref{case1:root} with representation \eqref{e0e1formula}, we get
\begin{align}\label{case1:e0def}
    \widehat{E}_{0}(t,\xi)&=\begin{cases}1 &\text{if } \xi\in\mathcal{R}_{1},\\
    \left[\cosh \left(\sqrt{\frac{1}{4}\lambda_{\xi}^{2\theta}-\lambda_{\xi}^{\sigma}} t\right)+\frac{\lambda_{\xi}^{\theta}\sinh \left(\sqrt{\frac{1}{4}\lambda_{\xi}^{2\theta}-\lambda_{\xi}^{\sigma}} t\right)}{2\sqrt{\frac{1}{4}\lambda_{\xi}^{2\theta}-\lambda_{\xi}^{\sigma}}} \right]e^{-\frac{1}{2}\lambda_{\xi}^{\theta}t}& \text {if } \xi\in\mathcal{R}_{2},\\
   \left[1+\frac{1}{2}\lambda_{\xi}^{\theta}t\right]e^{-\frac{1}{2}\lambda_{\xi}^{\theta}t}& \text {if } \xi\in\mathcal{R}_{3}, \\\left[ \cos \left(\sqrt{\lambda_{\xi}^{\sigma}-\frac{1}{4}\lambda_{\xi}^{2\theta}} t\right)+\frac{\lambda_{\xi}^{\theta}\sin\left(\sqrt{\lambda_{\xi}^{\sigma}-\frac{1}{4}\lambda_{\xi}^{2\theta}} t\right)}{2\sqrt{\lambda_{\xi}^{\sigma}-\frac{1}{4}\lambda_{\xi}^{2\theta}}}\right]e^{-\frac{1}{2}\lambda_{\xi}^{\theta}t}  & \text {if }  \xi\in\mathcal{R}_{4},\end{cases}
\end{align}
and
\begin{align}\label{case1:e1def}
    \widehat{E}_{1}(t,\xi)&=\begin{cases}t &\text{if } \xi\in\mathcal{R}_{1},\\
    \frac{\sinh \left(\sqrt{\frac{1}{4}\lambda_{\xi}^{2\theta}-\lambda_{\xi}^{\sigma}} t\right)}{\sqrt{\frac{1}{4}\lambda_{\xi}^{2\theta}-\lambda_{\xi}^{\sigma}}}e^{-\frac{1}{2}\lambda_{\xi}^{\theta}t} & \text {if } \xi\in\mathcal{R}_{2},\\ te^{-\frac{1}{2}\lambda_{\xi}^{\theta}t} & \text {if } \xi\in\mathcal{R}_{3},\\ \frac{\sin \left(\sqrt{\lambda_{\xi}^{\sigma}-\frac{1}{4}\lambda_{\xi}^{2\theta}} t\right)}{\sqrt{\lambda_{\xi}^{\sigma}-\frac{1}{4}\lambda_{\xi}^{2\theta}}}e^{-\frac{1}{2}\lambda_{\xi}^{\theta}t} & \text {if } \xi\in\mathcal{R}_{4}.\end{cases}
\end{align}
Now we will estimate $\widehat{E}_{0}(t,\xi)$ and $\widehat{E}_{1}(t,\xi)$ on $\mathcal{R}_{j}$ for $j=1,2,3,4$.

\medskip

\noindent\textbf{Estimate on $\mathcal{R}_{1}$:} From \eqref{case1:e0def} and \eqref{case1:e1def},  we have
\begin{equation*}
    |\widehat{E}_{0}(t,\xi)|= 1\quad\text{and}\quad     |\widehat{E}_{1}(t,\xi)|= t,\label{E_0 in R_1,1}\nonumber
\end{equation*}
and  hence the representation \eqref{usoln:fourier} gives
\begin{equation}\label{R11:usol}
    |\widehat{u}(t,\xi)|\lesssim |\widehat{u}_{0}(\xi)| +t|\widehat{u}_{1}(\xi)|.
\end{equation}
Furthermore, for any $\beta>0$ we have
\begin{equation}\label{R11:spacedersol}
    |\lambda_{\xi}^{\beta}\widehat{u}(t,\xi)|=0.
\end{equation}
Using the relations \eqref{timederive0} and \eqref{timederive1} for the equations \eqref{case1:e0def} and \eqref{case1:e1def}, respectively, we have
\begin{align}
   \partial^{k}_{t}\widehat{E}_{0}(t,\xi)=0\quad&\text{and}\quad \partial^{k}_{t}\widehat{E}_{1}(t,\xi)=   \begin{cases}
       1&\text{if } k=1,\\
       0&\text{if } k\geq 2,
   \end{cases}\nonumber
\end{align}
and hence the representation \eqref{timedersoln:fourier} along with the above relations gives
\begin{equation}\label{R11:timedersoln}
       |\partial_{t}\widehat{u}(t,\xi)|\lesssim |\widehat{u}_{1}(\xi)|\quad\text{and}\quad
    |\partial^{k}_{t}\widehat{u}(t,\xi)|=0\quad\text{for } k\geq 2.
\end{equation}
\noindent\textbf{Estimate on $\mathcal{R}_{2}$:} From \eqref{case1:e0def} and \eqref{case1:e1def}, we have
\begin{equation}
    |\widehat{E}_{0}(t,\xi)|\lesssim  \left[1+\frac{\lambda_{\xi}^{\theta}}{2\sqrt{\frac{1}{4}\lambda_{\xi}^{2\theta}-\lambda_{\xi}^{\sigma}}}\right]e^{\left(-\frac{1}{2}\lambda_{\xi}^{\theta}+\sqrt{\frac{1}{4}\lambda_{\xi}^{2\theta}-\lambda_{\xi}^{\sigma}}\right)t}\lesssim e^{-\lambda_{\xi}^{\sigma-\theta}t},\nonumber
\end{equation}
and
\begin{equation}
      |\widehat{E}_{1}(t,\xi)|\lesssim  \frac{1}{\sqrt{\frac{1}{4}\lambda_{\xi}^{2\theta}-\lambda_{\xi}^{\sigma}}}e^{\left(-\frac{1}{2}\lambda_{\xi}^{\theta}+\sqrt{\frac{1}{4}\lambda_{\xi}^{2\theta}-\lambda_{\xi}^{\sigma}}\right)t}\lesssim e^{-\lambda_{\xi}^{\sigma-\theta}t},\nonumber
\end{equation}
where we used the argument similar to $\widehat{E}_{0}$ and $\widehat{E}_{1}$ on $\mathcal{R}_{2}$. Consequently, the representation \eqref{usoln:fourier} along with the above estimates gives
\begin{equation}\label{R12:usoln}
    |\widehat{u}(t,\xi)|\lesssim e^{-\lambda_{\xi}^{\sigma-\theta}t}\left(|\widehat{u}_{0}(\xi)| +|\widehat{u}_{1}(\xi)|\right)\lesssim e^{-\delta t}\left(|\widehat{u}_{0}(\xi)| +|\widehat{u}_{1}(\xi)|\right),
\end{equation}
where the constant $\delta$ is obtained by using the relation \eqref{rmk3}.
Furthermore, for any $\beta>0$ we have
\begin{equation}\label{R12:spaceder}
\lambda_{\xi}^{\beta}|\widehat{u}(t,\xi)|\lesssim \lambda_{\xi}^{\beta}e^{-\lambda_{\xi}^{\sigma-\theta}t}\left(|\widehat{u}_{0}(t,\xi)|+|\widehat{u}_{1}(t,\xi)|\right)\lesssim e^{-\delta t}\left(|\widehat{u}_{0}(t,\xi)|+|\widehat{u}_{1}(t,\xi)|\right).
\end{equation}
Recalling the relation \eqref{timederive1} for $\widehat{E}_{1}(t,\xi)$, we have
\small{\begin{equation}
    \partial_{t}^{k}\widehat{E}_{1}(t,\xi)=\frac{\left(-\frac{1}{2}\lambda_{\xi}^{\theta}+\sqrt{\frac{1}{4}\lambda_{\xi}^{2\theta}-\lambda_{\xi}^{\sigma}}\right)^{k}}{\sqrt{\frac{1}{4}\lambda_{\xi}^{2\theta}-\lambda_{\xi}^{\sigma}}}e^{\left(-\frac{1}{2}\lambda_{\xi}^{\theta}+\sqrt{\frac{1}{4}\lambda_{\xi}^{2\theta}-\lambda_{\xi}^{\sigma}}\right)t}-\frac{\left(-\frac{1}{2}\lambda_{\xi}^{\theta}-\sqrt{\frac{1}{4}\lambda_{\xi}^{2\theta}-\lambda_{\xi}^{\sigma}}\right)^{k}}{\sqrt{\frac{1}{4}\lambda_{\xi}^{2\theta}-\lambda_{\xi}^{\sigma}}}e^{\left(-\frac{1}{2}\lambda_{\xi}^{\theta}-\sqrt{\frac{1}{4}\lambda_{\xi}^{2\theta}-\lambda_{\xi}^{\sigma}}\right)t}\nonumber
\end{equation}}
whence we get
\begin{equation}
    |\partial_{t}^{k}\widehat{E}_{1}(t,\xi)|\lesssim e^{-\lambda_{\xi}^{\sigma-\theta}t}+e^{-\frac{1}{2}\lambda_{\xi}^{\theta}t}\lesssim  e^{-\delta t},\nonumber
\end{equation}
where we used the argument similar to the estimate of $\partial_{t}^{k}\widehat{E}_{1}$ on $\mathcal{R}_{2}$ in Theorem \ref{thm:case0new} and
 consequently, using the relation \eqref{timederive0}, we obtain
\begin{equation*}
    |\partial_{t}^{k}\widehat{E}_{0}(t,\xi)|\lesssim \lambda_{\xi}^{\sigma}e^{-\delta t}
    \lesssim e^{-\delta t}.
\end{equation*}
Combining the above estimate with the representation \eqref{timedersoln:fourier}, we get
\begin{equation}\label{R12:timedersoln}
    |\partial_{t}^{k}\widehat{u}(t,\xi)|\lesssim e^{-\delta t} \left(|\widehat{u}_{0}(t,\xi)|+|\widehat{u}_{1}(t,\xi)|\right).
\end{equation}
\noindent\textbf{Estimate on $\mathcal{R}_{3}$:} From \eqref{case1:e0def} and \eqref{case1:e1def}, we have
\begin{equation}
     |\widehat{E}_{0}(t,\xi)|= \left(1+2^{\frac{\sigma}{2\theta-\sigma}}t\right)e^{-2^{\frac{\sigma}{2\theta-\sigma}}t}\lesssim(1+t)e^{-\delta t}\quad\text{and}\quad     |\widehat{E}_{1}(t,\xi)|= te^{-2^{\frac{\sigma}{2\theta-\sigma}}t}\lesssim te^{-\delta t},\nonumber
\end{equation}
this gives
\begin{equation}\label{R13:usoln}
    |\widehat{u}(t,\xi)|\lesssim (1+t) e^{-\delta t}|\widehat{u}_{0}(\xi)| +t e^{-\delta t}|\widehat{u}_{1}(\xi)|.
\end{equation}
Consequently, for any $\beta>0$, 
\begin{equation}\label{R13:spacedersoln}
    \lambda_{\xi}^{\beta}|\widehat{u}(t,\xi)|\lesssim (1+t)e^{-\delta t}|\widehat{u}_{0}(\xi)| +te^{-\delta t}|\widehat{u}_{1}(\xi)|,
\end{equation}
Similar to the estimates of $\partial^{k}_{t}\widehat{E}_{1}$ and $\partial^{k}_{t}\widehat{E}_{0}$ on $\mathcal{R}_{3}$ in undamped case, we obtain
\begin{equation*}
    |\partial^{k}_{t}\widehat{E}_{1}(t,\xi)|\lesssim (1+t)e^{-\delta t}\quad \text{and} \quad |\partial^{k}_{t}\widehat{E}_{0}(t,\xi)|\lesssim (1+t)e^{-\delta t}.
\end{equation*}
Combining the above estimate with the representation \eqref{timedersoln:fourier}, we have
\begin{equation}\label{R13:timedersoln}
    |\partial_{t}^{k}\widehat{u}(t,\xi)|\lesssim (1+t)e^{-\delta t}\left(|\widehat{u}_{0}(\xi)| +|\widehat{u}_{1}(\xi)|\right),\quad \text{for } k\geq 1.
\end{equation}
\noindent\textbf{Estimate on $\mathcal{R}_{4}$:} From \eqref{case1:e0def} and \eqref{case1:e1def}, using the fact that $\frac{\lambda_{\xi}^{\theta}}{2\sqrt{\lambda_{\xi}^{\sigma}-\frac{1}{4}\lambda_{\xi}^{2\theta}}}$ is bounded, we get
\begin{align}
     |\widehat{E}_{0}(t,\xi)|\lesssim e^{-\frac{1}{2}\lambda_{\xi}^{\theta}t}\lesssim e^{-\delta t}\quad&\text{and}\quad     |\widehat{E}_{1}(t,\xi)|\lesssim  e^{-\frac{1}{2}\lambda_{\xi}^{\theta}t}\lesssim e^{-\delta t}.\nonumber
     \end{align}
Consequently, we get
\begin{equation}\label{R14:usoln}
    |\widehat{u}(t,\xi)|\lesssim e^{-\delta t}\left(|\widehat{u}_{0}(\xi)| +|\widehat{u}_{1}(\xi)|\right).
\end{equation}
Further, for any $\beta>0$, the term $|\lambda_{\xi}^{\beta}\widehat{E}_{0}(t,\xi)|$ and $|\lambda_{\xi}^{\beta}\widehat{E}_{1}(t,\xi)|$ can be estimated as
\begin{equation}\label{case1:spderM1}
     |\lambda_{\xi}^{\beta}\widehat{E}_{0}(t,\xi)|\lesssim \lambda_{\xi}^{\beta} e^{-\frac{1}{2}\lambda_{\xi}^{\theta}t}=\left(\lambda_{\xi}^{\theta}t\right)^{\frac{\beta}{\theta}} e^{-\frac{1}{2}\lambda_{\xi}^{\theta}t}t^{-\frac{\beta}{\theta}}
     \lesssim\begin{cases}
         t^{-\frac{\beta}{\theta}}&\text{if } t\in (0,1],\\
         e^{-\delta t}&\text{if } t\in [1,\infty),
     \end{cases}
\end{equation}
and
\begin{equation}\label{case1:spderM12}
     |\lambda_{\xi}^{\beta}\widehat{E}_{1}(t,\xi)|\lesssim \lambda_{\xi}^{\beta-\frac{\sigma}{2}} e^{-\frac{1}{2}\lambda_{\xi}^{\theta}t}=\left(\lambda_{\xi}^{\theta}t\right)^{\frac{\beta-\frac{\sigma}{2}}{\theta}} e^{-\frac{1}{2}\lambda_{\xi}^{\theta}t}t^{-\frac{\beta-\frac{\sigma}{2}}{\theta}}
     \lesssim\begin{cases}
         t^{-\frac{2\beta-\sigma}{2\theta}}&\text{if } t\in (0,1],\\
         e^{-\delta t}&\text{if } t\in [1,\infty),
     \end{cases}
\end{equation}
whenever $2\beta>\sigma$, while
\begin{equation}\label{case1:spderM12new}
     |\lambda_{\xi}^{\beta}\widehat{E}_{1}(t,\xi)|\lesssim \lambda_{\xi}^{\beta-\frac{\sigma}{2}} e^{-\frac{1}{2}\lambda_{\xi}^{\theta}t}\lesssim e^{-\frac{1}{2}\lambda_{\xi}^{\theta}t}
     \lesssim
         e^{-\delta t},
\end{equation}
whenever $2\beta\leq \sigma$, where we have ultilize the relation \eqref{rmk2}.
Recalling the relation \eqref{timederive1} for $\widehat{E}_{1}(t,\xi)$, we have
\begin{align}\label{R14:timedersoln}
    |\partial_{t}^{k}\widehat{E}_{1}(t,\xi)|
    &\lesssim \lambda_{\xi}^{\frac{k-1}{2}\sigma}e^{-\frac{1}{2}\lambda_{\xi}^{\theta}t}=\left(\lambda_{\xi}^{\theta}t\right)^{\frac{(k-1)\sigma}{2\theta}} e^{-\frac{1}{2}\lambda_{\xi}^{\theta}t}t^{-\frac{(k-1)\sigma}{2\theta}}\lesssim\begin{cases}
         t^{-\frac{(k-1)\sigma}{2\theta}}&\text{if } t\in (0,1],\\
         e^{-\delta t}&\text{if } t\in [1,\infty),
     \end{cases}
\end{align}
and
\begin{align}\label{R14:timedersoln2}
    |\partial_{t}^{k}\widehat{E}_{0}(t,\xi)|
    &\lesssim \lambda_{\xi}^{\frac{k}{2}\sigma}e^{-\frac{1}{2}\lambda_{\xi}^{\theta}t}=\left(\lambda_{\xi}^{\theta}t\right)^{\frac{k\sigma}{2\theta}} e^{-\frac{1}{2}\lambda_{\xi}^{\theta}t}t^{-\frac{k\sigma}{2\theta}}\lesssim\begin{cases}
         t^{-\frac{k\sigma}{2\theta}}&\text{if } t\in (0,1],\\
         e^{-\delta t}&\text{if } t\in [1,\infty).
     \end{cases}
\end{align}
We are now in position to compute the estimate for $\|u(t)\|_{\mathcal{H}}, \|\mathcal{L}^{\beta}u(t)\|_{\mathcal{H}},$ and $\|\partial_{t}^{k}u(t)\|_{\mathcal{H}}$.

\medskip

\noindent\textbf{Estimate for $\|u(t)\|_{\mathcal{H}}$:} Recalling the Plancherel formula \eqref{planchform} along with the estimates \eqref{R11:usol}, \eqref{R12:usoln}, \eqref{R13:usoln}, and \eqref{R14:usoln}, we have
\begin{align}\label{case1:uest}
\|u(t)\|_{\mathcal{H}}^{2}
&\lesssim \sum_{\xi\in \mathcal{R}_{1}}\left(|\widehat{u}_{0}(\xi)|^{2} +t^{2}|\widehat{u}_{1}(\xi)|^{2}\right)+\sum_{\xi\in \mathcal{R}_{2}}e^{-\delta t}\left(|\widehat{u}_{0}(\xi)|^{2} +|\widehat{u}_{1}(\xi)|^{2}\right)\nonumber\\
&+ \sum_{\xi\in \mathcal{R}_{3}}(1+t)^{2}e^{-\delta t}\left(|\widehat{u}_{0}(\xi)|^{2} +|\widehat{u}_{1}(\xi)|^{2}\right)+\sum_{\xi\in \mathcal{R}_{4}}e^{-\delta t}\left(|\widehat{u}_{0}(\xi)|^{2} +|\widehat{u}_{1}(\xi)|^{2}\right)\lesssim\|u_{0}\|^{2}_{\mathcal{H}}+t^{2}\|u_{1}\|^{2}_{\mathcal{H}}.
\end{align}

\medskip

\noindent\textbf{Estimate for $\|\mathcal{L}^{\beta}u(t)\|_{\mathcal{H}}$:} Similar to the above estimate using Plancherel formula along with the estimate
\eqref{R11:spacedersol}, \eqref{R12:spaceder}, \eqref{R13:spacedersoln}, \eqref{case1:spderM1}, \eqref{case1:spderM12} and \eqref{case1:spderM12new}, we obtain
\begin{align}\label{case1:spacederuest new}
\left\|\mathcal{L}^{\beta}u(t)\right\|_{\mathcal{H}}^{2}
&\lesssim \sum_{\xi\in \mathcal{R}_{2}}e^{-\delta t}\left(|\widehat{u}_{0}(\xi)|^{2} +|\widehat{u}_{1}(\xi)|^{2}\right)
+ \sum_{\xi\in \mathcal{R}_{3}}(1+t)^{2}e^{-\delta t}\left(|\widehat{u}_{0}(\xi)|^{2} +|\widehat{u}_{1}(\xi)|^{2}\right)\nonumber\\&+\sum_{\xi\in \mathcal{R}_{4}}\begin{cases}
\begin{cases}
        t^{-\frac{2\beta}{\theta}}|\widehat{u}_{0}(\xi)|^{2}+t^{-\frac{2\beta-\sigma}{\theta}}|\widehat{u}_{1}(\xi)|^{2}&\text{if } t\in (0,1],\\
        e^{-\delta t}\left(|\widehat{u}_{0}(\xi)|^{2}+|\widehat{u}_{1}(\xi)|^{2}\right)&\text{if } t\in [1,\infty).
    \end{cases}&\text{if } 2\beta>\sigma,\\
\begin{cases}
        t^{-\frac{2\beta}{\theta}}|\widehat{u}_{0}(\xi)|^{2}+e^{-\delta t}|\widehat{u}_{1}(\xi)|^{2}&\text{if } t\in (0,1],\\
        e^{-\delta t}\left(|\widehat{u}_{0}(\xi)|^{2}+|\widehat{u}_{1}(\xi)|^{2}\right)&\text{if } t\in [1,\infty),
    \end{cases}&\text{if } 2\beta\leq\sigma.
\end{cases}\nonumber\\
&\lesssim \begin{cases}
\begin{cases}
        t^{-\frac{2\beta}{\theta}}\|u_{0}\|_{\mathcal{H}}^{2}+t^{-\frac{2\beta-\sigma}{\theta}}\|u_{1}\|_{\mathcal{H}}^{2}&\text{if } t\in (0,1],\\
        e^{-\delta t}\left(\|u_{0}\|_{\mathcal{H}}^{2}+\|u_{1}\|_{\mathcal{H}}^{2}\right)&\text{if } t\in [1,\infty).
    \end{cases}&\text{if } 2\beta>\sigma,\\
\begin{cases}
        t^{-\frac{2\beta}{\theta}}\|u_{0}\|_{\mathcal{H}}^{2}+e^{-\delta t}\|u_{1}\|_{\mathcal{H}}^{2}&\text{if } t\in (0,1],\\
        e^{-\delta t}\left(\|u_{0}\|_{\mathcal{H}}^{2}+\|u_{1}\|_{\mathcal{H}}^{2}\right)&\text{if } t\in [1,\infty),
    \end{cases}&\text{if } 2\beta\leq\sigma.
\end{cases}
\end{align}

\medskip

\noindent\textbf{Estimate for $\|\partial_{t}^{k}u(t)\|_{\mathcal{H}}$:}
Combining the estimates \eqref{R11:timedersoln}, \eqref{R12:timedersoln}, \eqref{R13:timedersoln}, \eqref{R14:timedersoln} and \eqref{R14:timedersoln2}, we get
\begin{multline}\label{case1:timederuest}
\|\partial_{t}u(t)\|_{\mathcal{H}}^{2}
\lesssim \sum_{\xi\in \mathcal{R}_{1}}
    |\widehat{u}_{1}(\xi)|^{2}+\sum_{\xi\in \mathcal{R}_{2}}e^{-\delta t} \left(|\widehat{u}_{0}(\xi)|^{2} +|\widehat{u}_{1}(\xi)|^{2}\right)+ \sum_{\xi\in \mathcal{R}_{3}}(1+t)^{2}e^{-\delta t}\left(|\widehat{u}_{0}(\xi)|^{2} +|\widehat{u}_{1}(\xi)|^{2}\right)\\+\sum_{\xi\in \mathcal{R}_{4}}\begin{cases}
        t^{-\frac{\sigma}{\theta}}|\widehat{u}_{0}(\xi)|^{2}+|\widehat{u}_{1}(\xi)|^{2}&\text{if } t\in (0,1],\\
        e^{-\delta t}|\widehat{u}_{0}(\xi)|^{2}+|\widehat{u}_{1}(\xi)|^{2}&\text{if } t\in [1,\infty),
    \end{cases}\lesssim \begin{cases}
        t^{-\frac{\sigma}{\theta}}\|u_{0}\|_{\mathcal{H}}^{2}+\|u_{1}\|_{\mathcal{H}}^{2}&\text{if } t\in (0,1],\\
        e^{-\delta t}\|u_{0}\|_{\mathcal{H}}^{2}+\|u_{1}\|_{\mathcal{H}}^{2}&\text{if } t\in [1,\infty),
    \end{cases}
\end{multline} 
for some  $\delta>0$. Furthermore, for $k\geq 2$, consider 
\begin{multline}\label{case1:timederuestk}
\|\partial_{t}^{k}u(t)\|_{\mathcal{H}}^{2}
\lesssim \sum_{\xi\in \mathcal{R}_{2}}e^{-\delta t} \left(|\widehat{u}_{0}(\xi)|^{2} +|\widehat{u}_{1}(\xi)|^{2}\right)+ \sum_{\xi\in \mathcal{R}_{3}}(1+t)^{2}e^{-\delta t}\left(|\widehat{u}_{0}(\xi)|^{2} +|\widehat{u}_{1}(\xi)|^{2}\right)\\+\sum_{\xi\in \mathcal{R}_{4}}\begin{cases}
 t^{-\frac{k\sigma}{\theta}}|\widehat{u}_{0}(\xi)|^{2}+t^{-\frac{(k-1)\sigma}{\theta}}|\widehat{u}_{1}(\xi)|^{2}&\text{if } t\in (0,1],\\
   e^{-\delta t}\left(|\widehat{u}_{0}(\xi)|^{2}+|\widehat{u}_{1}(\xi)|^{2}\right)&\text{if } t\in [1,\infty),
\end{cases}\lesssim
\begin{cases}
 t^{-\frac{k\sigma}{\theta}}\|u_{0}\|_{\mathcal{H}}^{2}+t^{-\frac{(k-1)\sigma}{\theta}}\|u_{1}\|_{\mathcal{H}}^{2}&\text{if } t\in (0,1],\\
   e^{-\delta t}\left(\|u_{0}\|_{\mathcal{H}}^{2}+\|u_{1}\|_{\mathcal{H}}^{2}\right)&\text{if } t\in [1,\infty),
\end{cases}.
\end{multline}
This completes the proof.
\end{proof}
\begin{remark}
    In contrast to the previous instance, where exponential decay was accomplished with Sobolev regularity on the initial data, in this situation we are able to eliminate the $\lambda_{\xi}$ factor in the region $\mathcal{R}_{4}$ by using the relation \eqref{rmk2}. 
\end{remark}
Having established the decay estimates in the case of effective damping, we can now obtain improved decay estimates with Sobolev regularity on the initial Cauchy data, as stated in Theorem \ref{thm:case0old} for effective damping, with some modifications outlined in the following remark.
\begin{remark}\label{rmk:eff-damping}
Note that the estimates \eqref{case1:spderM1} and \eqref{case1:spderM12} on $\mathcal{R}_{4}$ can also be estimated as
\begin{equation}\label{case1:spderM1modify}
     |\lambda_{\xi}^{\beta}\widehat{E}_{0}(t,\xi)|\lesssim \lambda_{\xi}^{\beta}e^{-\delta t}\quad \text{and}\quad |\lambda_{\xi}^{\beta}\widehat{E}_{1}(t,\xi)|\lesssim  \frac{\lambda_{\xi}^{\beta}}{2\sqrt{\lambda_{\xi}^{\sigma}-\frac{1}{4}\lambda_{\xi}^{2\theta}}}e^{-\delta t}\lesssim \lambda_{\xi}^{\beta-\frac{\sigma}{2}}e^{-\delta t}.
\end{equation}
Consequently, using the estimates \eqref{case1:spderM1modify} with \eqref{usoln:fourier} we obtain
\begin{equation}
|\lambda_{\xi}^{\beta}\widehat{u}(t,\xi)|\lesssim e^{-\delta t}
\left(\lambda_{\xi}^{\beta}|\widehat{u}_{0}(\xi)|+\lambda_{\xi}^{\beta-\frac{\sigma}{2}}|\widehat{u}_{1}(\xi)|\right).\nonumber
\end{equation}
Hence, the estimate of $\|\mathcal{L}^{\beta}u\|_{\mathcal{H}}$ will be replaced by
\begin{equation}
    \|\mathcal{L}^{\beta}u\|_{\mathcal{H}}^{2}\lesssim e^{-\delta t}\begin{cases}
\left(\|u_{0}\|^{2}_{\mathcal{H}^{2\beta}_{\mathcal{L}}}+\|u_{1}\|^{2}_{\mathcal{H}^{2\beta-\sigma}_{\mathcal{L}}}\right)&\text{if } 2\beta>\sigma,\\
\left(\|u_{0}\|^{2}_{\mathcal{H}^{2\beta}_{\mathcal{L}}}+\|u_{1}\|^{2}_{\mathcal{H}}\right)&\text{if } 2\beta\leq\sigma.
    \end{cases}\nonumber
\end{equation}
 Similarly, the estimate \eqref{R14:timedersoln} and \eqref{R14:timedersoln2} can also be estimated as
\begin{align}
    |\partial_{t}^{k}\widehat{E}_{1}(t,\xi)|\lesssim \lambda_{\xi}^{\frac{k-1}{2}\sigma}e^{-\delta t}\quad\text{and}\quad |\partial_{t}^{k}\widehat{E}_{0}(t,\xi)|\lesssim  \lambda_{\xi}^{\frac{k}{2}\sigma}e^{-\delta t}.\nonumber
\end{align}
So, we obtain
\begin{equation}
      |\partial_{t}^{k}\widehat{u}(t,\xi)|\lesssim e^{-\delta t}\left(\lambda_{\xi}^{\frac{k}{2}\sigma}|\widehat{u}_{0}(\xi)| +\lambda_{\xi}^{\frac{k-1}{2}\sigma}|\widehat{u}_{1}(\xi)|\right),\quad \text{for } k\geq 1,\nonumber
\end{equation}
and hence the estimate \eqref{case1:timederuest} and \eqref{case1:timederuestk} will be replaced by
\begin{equation}
    \|\partial_{t}u\|_{\mathcal{H}}^{2}\lesssim e^{-\delta t}\left\|u_{0}\right\|^{2}_{\mathcal{H}^{\sigma}_{\mathcal{L}}}+\|u_{1}\|^{2}_{\mathcal{H}},\nonumber
\end{equation}
and 
\begin{align}
\|\partial_{t}^{k}u(t)\|_{\mathcal{H}}^{2}&\lesssim
e^{-\delta t}\left(\|u_{0}\|^{2}_{\mathcal{H}^{k\sigma}_{\mathcal{L}}}+\|u_{1}\|^{2}_{\mathcal{H}^{(k-1)\sigma}_{\mathcal{L}}}\right)\nonumber
\end{align}
for $k\geq 2$. This observation proves Theorem \ref{thm:case0old} for the case $2\theta \in (0, \sigma)$.
\end{remark}
Let us now begin to prove the decay estimates in the critical case.
\begin{proof}[Proof of Theorem \ref{thm:case2}]
In this case it is enough to consider the following partition of $\mathcal{I}$:
 
    \begin{enumerate}
    \item $\mathcal{R}_{1}=\{\xi\in \mathcal{I}: \lambda_{\xi}=0\}$, and
    \item $\mathcal{R}_{2}=\left\{\xi\in \mathcal{I}: \lambda_{\xi}> 0\right\}$.
\end{enumerate}
Combining the relation \eqref{case2:root} with representation \eqref{e0e1formula}, we get
\begin{equation}
    \widehat{E}_{0}(t,\xi)=\begin{cases}1 &\text{if } \xi\in\mathcal{R}_{1},\\
    \left[ \cos \left(\frac{\sqrt{3}}{2}\lambda_{\xi}^{\theta}t\right)+\frac{1}{\sqrt{3}}\sin\left(\frac{\sqrt{3}}{2}\lambda_{\xi}^{\theta} t\right)\right]e^{-\frac{1}{2}\lambda_{\xi}^{\theta}t} & \text {if } \xi\in\mathcal{R}_{2},\end{cases}\label{E_0 in R^2_2}
\end{equation}
and
\begin{align}\label{E_1 in R^2_2}
    \widehat{E}_{1}(t,\xi)
    =\begin{cases}t &\text{if } \xi\in\mathcal{R}_{1},\\
    \frac{\sin \left(\sqrt{\frac{3}{4}}\lambda_{\xi}^{\theta} t\right)}{\sqrt{3}\lambda_{\xi}^{\theta}}e^{-\frac{1}{2}\lambda_{\xi}^{\theta}t}  & \text {if } \xi\in\mathcal{R}_{2}.
    \end{cases}
\end{align}
Now we will estimate $\widehat{E}_{0}(t,\xi)$ and $\widehat{E}_{1}(t,\xi)$ on $\mathcal{R}_{j}$ for $j=1,2$.
\medskip

\noindent\textbf{Estimate on $\mathcal{R}_{1}$:} Since the expression of $\widehat{E}_{0}(t,\xi)$ and $\widehat{E}_{1}(t,\xi)$ on $\mathcal{R}_{1}$ in this case are exactly same as in the case of effective damping, therefore the estimates for $\widehat{u}(t,\xi)$, $\lambda_{\xi}^{\beta}\widehat{u}(t,\xi)$ and $\partial_{t}^{k}\widehat{u}(t,\xi)$ in this case are exactly equal to \eqref{R11:usol}, \eqref{R11:spacedersol}, and \eqref{R11:timedersoln}, respectively. 
\medskip

\noindent\textbf{Estimate on $\mathcal{R}_{2}$:}
From  \eqref{E_0 in R^2_2} and \eqref{E_1 in R^2_2}, one get
\begin{align}
     |\widehat{E}_{0}(t,\xi)|\lesssim e^{-\delta t}\quad&\text{and}\quad     |\widehat{E}_{1}(t,\xi)|\lesssim e^{-\delta t}.\nonumber
     \end{align}
     and  hence the representation \eqref{usoln:fourier} gives
\begin{equation}\label{R22:usoln}
    |\widehat{u}(t,\xi)|\lesssim e^{-\delta t}\left(|\widehat{u}_{0}(\xi)| +|\widehat{u}_{1}(\xi)|\right).
\end{equation}
Further, for any $\beta>0$ we have
\begin{equation}\label{derivate in beta R^2_2}
     |\lambda_{\xi}^{\beta}\widehat{E}_{0}(t,\xi)|\lesssim \lambda_{\xi}^{\beta} e^{-\frac{1}{2}\lambda_{\xi}^{\theta}t}=\left(\lambda_{\xi}^{\theta}t\right)^{\frac{\beta}{\theta}} e^{-\frac{1}{2}\lambda_{\xi}^{\theta}t}t^{-\frac{\beta}{\theta}}
     \lesssim\begin{cases}
         t^{-\frac{\beta}{\theta}}&\text{if } t\in (0,1],\\
         e^{-\delta t}&\text{if } t\in [1,\infty),
     \end{cases}
\end{equation}
and
\begin{equation}\label{derivate in beta R^2_2 2nd}
     |\lambda_{\xi}^{\beta}\widehat{E}_{1}(t,\xi)|\lesssim \lambda_{\xi}^{\beta-\theta} e^{-\frac{1}{2}\lambda_{\xi}^{\theta}t}=\left(\lambda_{\xi}^{\theta}t\right)^{\frac{\beta-\theta}{\theta}} e^{-\frac{1}{2}\lambda_{\xi}^{\theta}t}t^{-\frac{\beta-\theta}{\theta}}
     \lesssim\begin{cases}
         t^{-\frac{\beta-\theta}{\theta}}&\text{if } t\in (0,1],\\
         e^{-\delta t}&\text{if } t\in [1,\infty),
     \end{cases}
\end{equation}
whenever $\beta> \theta$, while
\begin{equation*}
     |\lambda_{\xi}^{\beta}\widehat{E}_{1}(t,\xi)|\lesssim \lambda_{\xi}^{\beta-\theta} e^{-\frac{1}{2}\lambda_{\xi}^{\theta}t}\lesssim e^{-\frac{1}{2}\lambda_{\xi}^{\theta}t}
     \lesssim
         e^{-\delta t},
\end{equation*}
whenever $\beta\leq \theta$.
Recalling the relation \eqref{timederive1} for $\widehat{E}_{1}(t,\xi)$ on $\mathcal{R}_{2}$, we have
     \begin{align}\label{E31derivative}
    |\partial_{t}^{k}\widehat{E}_{1}(t,\xi)|
    &\lesssim \lambda_{\xi}^{(k-1)\theta}e^{-\frac{1}{2}\lambda_{\xi}^{\theta}t}=\left(\lambda_{\xi}^{\theta}t\right)^{k-1} e^{-\frac{1}{2}\lambda_{\xi}^{\theta}t}t^{-(k-1)}\lesssim\begin{cases}
         t^{-(k-1)}&\text{if } t\in (0,1],\\
         e^{-\delta t}&\text{if } t\in [1,\infty),
     \end{cases}
\end{align}
and
\begin{align}\label{E30derivative}
    |\partial_{t}^{k}\widehat{E}_{0}(t,\xi)|
    &\lesssim \lambda_{\xi}^{k\theta}e^{-\frac{1}{2}\lambda_{\xi}^{\theta}t}=\left(\lambda_{\xi}^{\theta}t\right)^{k} e^{-\frac{1}{2}\lambda_{\xi}^{\theta}t}t^{-k}\lesssim\begin{cases}
         t^{-k}&\text{if } t\in (0,1],\\
         e^{-\delta t}&\text{if } t\in [1,\infty).
     \end{cases}
\end{align}
We are now in position to develop the estimate $\|u\|_{\mathcal{H}},\|\mathcal{L}^{\beta}u\|_{\mathcal{H}}$, and $\|\partial_{t}^{k}u\|_{\mathcal{H}}$.
\medskip

\noindent\textbf{Estimate for $\|u(t)\|_{\mathcal{H}}$:} Similar to the estimate \eqref{case1:uest}, combining the estimate \eqref{R11:usol} and \eqref{R22:usoln}, we get
\begin{align}
\|u(t)\|_{\mathcal{H}}^{2}
&\lesssim \sum_{\xi\in \mathcal{R}_{1}}\left(|\widehat{u}_{0}(\xi)|^{2} +t^{2}|\widehat{u}_{1}(\xi)|^{2}\right)+\sum_{\xi\in \mathcal{R}_{2}}e^{-\delta t}\left(|\widehat{u}_{0}(\xi)|^{2} +|\widehat{u}_{1}(\xi)|^{2}\right)\lesssim\|u_{0}\|^{2}_{\mathcal{H}}+t^{2}\|u_{1}\|^{2}_{\mathcal{H}}.\nonumber
\end{align}
\textbf{Estimate for $\|\mathcal{L}^{\beta}u(t)\|_{\mathcal{H}}$:} Similar to the estimate \eqref{case1:spacederuest new}, combining the estimate \eqref{R11:spacedersol}, \eqref{derivate in beta R^2_2}, and \eqref{derivate in beta R^2_2 2nd},
 we get
\begin{align}\label{case2:spacederuest}
\left\|\mathcal{L}^{\beta}u(t)\right\|_{\mathcal{H}}^{2}&\lesssim \sum_{\xi\in \mathcal{R}_{2}}\begin{cases}
\begin{cases}
        t^{-\frac{2\beta}{\theta}}|\widehat{u}_{0}(\xi)|^{2}+t^{-\frac{2(\beta-\theta)}{\theta}}|\widehat{u}_{1}(\xi)|^{2}&\text{if } t\in (0,1],\\
        e^{-\delta t}\left(|\widehat{u}_{0}(\xi)|^{2}+|\widehat{u}_{1}(\xi)|^{2}\right)&\text{if } t\in [1,\infty).
    \end{cases}&\text{if } \beta> \theta,\\
\begin{cases}
        t^{-\frac{2\beta}{\theta}}|\widehat{u}_{0}(\xi)|^{2}+e^{-\delta t}|\widehat{u}_{1}(\xi)|^{2}&\text{if } t\in (0,1],\\
        e^{-\delta t}\left(|\widehat{u}_{0}(\xi)|^{2}+|\widehat{u}_{1}(\xi)|^{2}\right)&\text{if } t\in [1,\infty),
    \end{cases}&\text{if } \beta\leq \theta.
\end{cases}\nonumber\\
&\lesssim \begin{cases}
\begin{cases}
        t^{-\frac{2\beta}{\theta}}\|u_{0}\|_{\mathcal{H}}^{2}+t^{-\frac{2(\beta-\theta)}{\theta}}\|u_{1}\|_{\mathcal{H}}^{2}&\text{if } t\in (0,1],\\
        e^{-\delta t}\left(\|u_{0}\|_{\mathcal{H}}^{2}+\|u_{1}\|_{\mathcal{H}}^{2}\right)&\text{if } t\in [1,\infty).
    \end{cases}&\text{if } \beta> \theta,\\
\begin{cases}
        t^{-\frac{2\beta}{\theta}}\|u_{0}\|_{\mathcal{H}}^{2}+e^{-\delta t}\|u_{1}\|_{\mathcal{H}}^{2}&\text{if } t\in (0,1],\\
        e^{-\delta t}\left(\|u_{0}\|_{\mathcal{H}}^{2}+\|u_{1}\|_{\mathcal{H}}^{2}\right)&\text{if } t\in [1,\infty),
    \end{cases}&\text{if } \beta\leq \theta.
\end{cases}
\end{align}
\textbf{Estimate for $\|\partial_{t}^{k}u(t)\|_{\mathcal{H}}$:} Similar to the estimate \eqref{case1:timederuest}, combining the estimate \eqref{R11:timedersoln}, \eqref{E31derivative}, and \eqref{E30derivative}, we get
\begin{align}\label{case2:timederuest}
\|\partial_{t}u(t)\|_{\mathcal{H}}^{2}
&\lesssim \sum_{\xi\in \mathcal{R}_{1}}
    |\widehat{u}_{1}(\xi)|^{2}
+\sum_{\xi\in \mathcal{R}_{2}}\begin{cases}
t^{-2}|\widehat{u}_{0}(\xi)|^{2}+|\widehat{u}_{1}(\xi)|^{2}&\text{if } t\in (0,1],\\
        e^{-\delta t}|\widehat{u}_{0}(\xi)|^{2}+|\widehat{u}_{1}(\xi)|^{2}&\text{if } t\in [1,\infty).
    \end{cases}\nonumber\\
&\lesssim \begin{cases}
t^{-2}\|u_{0}\|_{\mathcal{H}}^{2}+\|u_{1}\|_{\mathcal{H}}^{2}&\text{if } t\in (0,1],\\
        e^{-\delta t}\|u_{0}\|_{\mathcal{H}}^{2}+\|u_{1}\|_{\mathcal{H}}^{2}&\text{if } t\in [1,\infty).
    \end{cases}
\end{align}
Furthermore, for $k\geq 2$, we have
\begin{align}\label{case2:timederuestk}
\|\partial_{t}^{k}u(t)\|_{\mathcal{H}}^{2}
&\lesssim \sum_{\xi\in \mathcal{R}_{2}}\begin{cases}
 t^{-2k}|\widehat{u}_{0}(\xi)|^{2}+t^{-2(k-1)}|\widehat{u}_{1}(\xi)|^{2}&\text{if } t\in (0,1],\\
   e^{-\delta t}\left(|\widehat{u}_{0}(\xi)|^{2}+|\widehat{u}_{1}(\xi)|^{2}\right)&\text{if } t\in [1,\infty),
\end{cases}\nonumber\\
&\lesssim \begin{cases}
 t^{-2k}\|u_{0}\|_{\mathcal{H}}+t^{-2(k-1)}\|u_{1}\|_{\mathcal{H}}^{2}&\text{if } t\in (0,1],\\
   e^{-\delta t}\left(\|u_{0}\|_{\mathcal{H}}^{2}+\|u_{1}\|_{\mathcal{H}}^{2}\right)&\text{if } t\in [1,\infty),
\end{cases}
\end{align}
This completes the proof of Theorem \ref{thm:case2}.
\end{proof}
Upon establishing the decay estimates for the  case of critical damping, we may now demonstrate Theorem \ref{thm:case0old} for the critical case, with certain adjustments outlined in the subsequent remark.
\begin{remark}\label{rmk:critical-damping}
The estimates \eqref{derivate in beta R^2_2} and \eqref{derivate in beta R^2_2 2nd} on $\mathcal{R}_{2}$ can also be estimated as
\begin{align}
     |\lambda_{\xi}^{\beta}\widehat{E}_{0}(t,\xi)|\lesssim \lambda_{\xi}^{\beta}e^{-\delta t}\quad \text{and}\quad
     |\lambda_{\xi}^{\beta}\widehat{E}_{1}(t,\xi)|\lesssim  e^{-\delta t}\frac{\lambda_{\xi}^{\beta}}{\lambda_{\xi}^{\theta}}\lesssim e^{-\delta t}\lambda_{\xi}^{\beta-\frac{\sigma}{2}}.\nonumber
\end{align}
Consequently, we can estimate
\begin{equation}
         \lambda_{\xi}^{\beta}|\widehat{u}(t,\xi)|\lesssim e^{-\delta t}\left(\lambda_{\xi}^{\beta}|\widehat{u}_{0}(\xi)| +\lambda_{\xi}^{\beta-\frac{\sigma}{2}}|\widehat{u}_{1}(t,\xi)|\right).\nonumber
     \end{equation}
     Hence, the estimate \eqref{case2:spacederuest}  will be replaced by \eqref{case0:spacedernorm new}.
      Similarly, the estimate \eqref{E31derivative} and \eqref{E30derivative} can also be estimated as
      \begin{align}
    |\partial_{t}^{k}\widehat{E}_{1}(t,\xi)|\lesssim \lambda_{\xi}^{\frac{(k-1)\sigma}{2}}e^{-\delta t}\quad\text{and}\quad |\partial_{t}^{k}\widehat{E}_{0}(t,\xi)|\lesssim  \lambda_{\xi}^{\frac{k\sigma}{2}}e^{-\delta t}.\nonumber
\end{align}
This gives
\begin{equation}
      |\partial_{t}^{k}\widehat{u}(t,\xi)|\lesssim e^{-\delta t}\left(\lambda_{\xi}^{\frac{k\sigma}{2}}|\widehat{u}_{0}(\xi)| +\lambda_{\xi}^{\frac{(k-1)\sigma}{2}}|\widehat{u}_{1}(\xi)|\right),\quad \text{for } k\geq 1,\nonumber
\end{equation}
and hence the estimate \eqref{case2:timederuest} and \eqref{case2:timederuestk} will be replaced by \eqref{case0:timedersoln:k>1 new}. This completes the proof of Theorem \ref{thm:case0old} for the case $2\theta =\sigma$.
\end{remark}
Let us now proceed to prove the decay estimates in the case of non-effective damping.
\begin{proof}[Proof of Theorem \ref{thm:case3}]
Similar to the case of effective damping, it is reasonable to consider the following partition of $\mathcal{I}$ in this case:
 
    \begin{enumerate}
    \item $\mathcal{R}_{1}=\{\xi\in \mathcal{I}: \lambda_{\xi}^{2\theta-\sigma}=0\},$
    \item  $\mathcal{R}_{2}=\left\{\xi\in \mathcal{I}: 0<\lambda_{\xi}^{2\theta-\sigma}<4\right\}$,
    \item $\mathcal{R}_{3}=\left\{\xi\in \mathcal{I}: \lambda_{\xi}^{2\theta-\sigma}=4\right\}$, and
    \item $ \mathcal{R}_{4}=\left\{\xi\in \mathcal{I}: \lambda_{\xi}^{2\theta-\sigma}>4\right\}$.
\end{enumerate}
Combining the relation \eqref{case3:root} with representation \eqref{e0e1formula}, we get
\begin{equation}\label{case3:e0def}
    \widehat{E}_{0}(t,\xi)=\begin{cases}1 &\text{if } \xi\in\mathcal{R}_{1},\\
    \left[\cos \left(\frac{1}{2}\sqrt{4\lambda_{\xi}^{\sigma}-\lambda_{\xi}^{2\theta}} t\right)+\frac{\lambda_{\xi}^{\theta}\sin\left(\frac{1}{2}\sqrt{4\lambda_{\xi}^{\sigma}-\lambda_{\xi}^{2\theta}} t\right)}{\sqrt{4\lambda_{\xi}^{\sigma}-\lambda_{\xi}^{2\theta}}}  \right]e^{-\frac{1}{2}\lambda_{\xi}^{\theta}t}& \text {if } \xi\in\mathcal{R}_{2},\\
   \left[1+\frac{1}{2}\lambda_{\xi}^{\theta}t\right]e^{-\frac{1}{2}\lambda_{\xi}^{\theta}t}& \text {if } \xi\in\mathcal{R}_{3}, \\
   \left[
    \cosh \left(\frac{1}{2}\sqrt{\lambda_{\xi}^{2\theta}-4\lambda_{\xi}^{\sigma}} t\right)+\frac{\lambda_{\xi}^{\theta}\sinh \left(\frac{1}{2}\sqrt{\lambda_{\xi}^{2\theta}-4\lambda_{\xi}^{\sigma}} t\right)}{\sqrt{\lambda_{\xi}^{2\theta}-4\lambda_{\xi}^{\sigma}}}\right]e^{-\frac{1}{2}\lambda_{\xi}^{\theta}t} & \text {if }  \xi\in\mathcal{R}_{4},\end{cases}
\end{equation}
and
\begin{equation}\label{case3:e1def}
    \widehat{E}_{1}(t,\xi)=\begin{cases}t &\text{if } \xi\in\mathcal{R}_{1},\\
    \frac{\sin \left(\frac{1}{2}\sqrt{4\lambda_{\xi}^{\sigma}-\lambda_{\xi}^{2\theta}} t\right)}{\sqrt{4\lambda_{\xi}^{\sigma}-\lambda_{\xi}^{2\theta}}} e^{-\frac{1}{2}\lambda_{\xi}^{\theta}t}& \text {if } \xi\in\mathcal{R}_{2},\\ te^{-\frac{1}{2}\lambda_{\xi}^{\theta}t} & \text {if } \xi\in\mathcal{R}_{3},\\ \frac{\sinh \left(\frac{1}{2}\sqrt{\lambda_{\xi}^{2\theta}-4\lambda_{\xi}^{\sigma}} t\right)}{\sqrt{\lambda_{\xi}^{2\theta}-4\lambda_{\xi}^{\sigma}}}e^{-\frac{1}{2}\lambda_{\xi}^{\theta}t} & \text {if } \xi\in\mathcal{R}_{4}.\end{cases}
\end{equation}
Now we will estimate $\widehat{E}_{0}(t,\xi)$ and $\widehat{E}_{1}(t,\xi)$ on $\mathcal{R}_{j}$ for $j=1,2,3,4$.

\medskip

\noindent\textbf{Estimate on $\mathcal{R}_{1}$:} Since $\widehat{E}_{0}(t,\xi)$   and $\widehat{E}_{1}(t,\xi)$ have the same expression as in the effective damping case on $\mathcal{R}_{1}$, therefore the estimates for $\widehat{u}(t,\xi)$, $\lambda_{\xi}^{\beta}\widehat{u}(t,\xi)$ and $\partial_{t}^{k}\widehat{u}(t,\xi)$ in this case are exactly equal to \eqref{R11:usol}, \eqref{R11:spacedersol}, and \eqref{R11:timedersoln}, respectively.

\noindent\textbf{Estimate on $\mathcal{R}_{2}$:} 
From \eqref{case3:e0def} and \eqref{case3:e1def}, using the fact that $\frac{\lambda_{\xi}^{\theta}}{\sqrt{4\lambda_{\xi}^{\sigma}-\lambda_{\xi}^{2\theta}}}$ is bounded and $\left|\sin x\right|\leq 1$, we get
\begin{align}\label{estimate of u case4_2}
     |\widehat{E}_{0}(t,\xi)|,~|\widehat{E}_{1}(t,\xi)|\lesssim e^{-\delta t}\quad \text{and hence we get }|\widehat{u}(t,\xi)|\lesssim e^{-\delta t}\left(|\widehat{u}_{0}(\xi)| +|\widehat{u}_{1}(\xi)|\right).
     \end{align}
Further, for any $\beta>0$ we have
\begin{align}
     |\lambda_{\xi}^{\beta}\widehat{E}_{0}(t,\xi)|\lesssim \lambda_{\xi}^{\beta}e^{-\frac{1}{2}\lambda_{\xi}^{\theta}t}\lesssim e^{-\delta t}\quad \text{and}\quad
|\lambda_{\xi}^{\beta}\widehat{E}_{1}(t,\xi)|\lesssim e^{-\delta t}.\nonumber
\end{align}
Consequently, we can estimate
\begin{equation}\label{estimate of space derivative case3 R_2}
|\lambda_{\xi}^{\beta}\widehat{u}(t,\xi)|\lesssim e^{-\delta t}
\left(|\widehat{u}_{0}(\xi)|+|\widehat{u}_{1}(\xi)|\right).
\end{equation}
Recalling the relation \eqref{timederive1} for $\widehat{E}_{1}(t,\xi)$ on $\mathcal{R}_{3}$, we have
\begin{multline}
    \partial_{t}^{k}\widehat{E}_{1}(t,\xi)=\\\frac{\left(-\frac{1}{2}\lambda_{\xi}^{\theta}+\frac{i}{2}\sqrt{4\lambda_{\xi}^{\sigma}-\lambda_{\xi}^{2\theta}}\right)^{k}}{i\sqrt{4\lambda_{\xi}^{\sigma}-\lambda_{\xi}^{2\theta}}}e^{\left(-\frac{1}{2}\lambda_{\xi}^{\theta}+\frac{i}{2}\sqrt{4\lambda_{\xi}^{\sigma}-\lambda_{\xi}^{2\theta}}\right)t}-\frac{\left(-\frac{1}{2}\lambda_{\xi}^{\theta}-\frac{i}{2}\sqrt{4\lambda_{\xi}^{\sigma}-\lambda_{\xi}^{2\theta}}\right)^{k}}{i\sqrt{4\lambda_{\xi}^{\sigma}-\lambda_{\xi}^{2\theta}}}e^{\left(-\frac{1}{2}\lambda_{\xi}^{\theta}-\frac{i}{2}\sqrt{4\lambda_{\xi}^{\sigma}-\lambda_{\xi}^{2\theta}}\right)t}.\nonumber
\end{multline}
Using the boundedness of $\frac{1}{\sqrt{4\lambda_{\xi}^{\sigma}-\lambda_{\xi}^{2\theta}}}$, we obtain
\begin{align*}
    |\partial_{t}^{k}\widehat{E}_{1}(t,\xi)|&\lesssim\left(\left|\left(-\frac{1}{2}\lambda_{\xi}^{\theta}+\frac{i}{2}\sqrt{4\lambda_{\xi}^{\sigma}-\lambda_{\xi}^{2\theta}}\right)^{k}\right|+\left|\left(-\frac{1}{2}\lambda_{\xi}^{\theta}-\frac{i}{2}\sqrt{4\lambda_{\xi}^{\sigma}-\lambda_{\xi}^{2\theta}}\right)^{k}\right|\right)e^{-\frac{1}{2}\lambda_{\xi}^{\theta}t} \lesssim e^{-\delta t},
\end{align*}
 Consequently using \eqref{timederive0}, we get
\begin{equation}
    |\partial_{t}^{k}\widehat{E}_{0}(t,\xi)|\lesssim \lambda_{\xi}^{\sigma}|\partial_{t}^{k-1}\widehat{E}_{1}(t,\xi)|\lesssim \lambda_{\xi}^{\frac{k+1}{2}\sigma}e^{-\frac{1}{2}\lambda_{\xi}^{\theta}t}\lesssim e^{-\delta t}.\nonumber
\end{equation}
Hence we obtain
\begin{equation}\label{estimate of t derivative case 3 R_2}
      |\partial_{t}^{k}\widehat{u}(t,\xi)|\lesssim e^{-\delta t}\left(|\widehat{u}_{0}(\xi)| +|\widehat{u}_{1}(\xi)|\right),\quad \text{for } k\geq 1.
\end{equation}

\noindent\textbf{Estimate on $\mathcal{R}_{3}$:} From \eqref{case3:e0def} and \eqref{case3:e1def}, we have
\begin{equation}\label{estimate of u case4_3}
     |\widehat{E}_{0}(t,\xi)|\lesssim (1+t)e^{-\delta t}\quad\text{and}\quad     |\widehat{E}_{1}(t,\xi)|\lesssim te^{-\delta t}.\nonumber
\end{equation}
This gives
\begin{equation}\label{estimate of u case3 R_3}
    |\widehat{u}(t,\xi)|\lesssim (1+t)e^{-\delta t}|\widehat{u}_{0}(\xi)| +te^{-\delta t}\widehat{u}_{1}(\xi)|.
\end{equation}
Consequently, for any $\beta>0$, 
\begin{equation}\label{estimate of space derivative case3 R_3}
    \lambda_{\xi}^{\beta}|\widehat{u}(t,\xi)|\lesssim (1+t)e^{-\delta t}|\widehat{u}_{0}(\xi)| +te^{-\delta t}|\widehat{u}_{1}(\xi)|.
\end{equation}
Furthermore, using the relation \eqref{timederive1} and \eqref{case3:e1def}, we have
\begin{equation}
    |\partial^{k}_{t}\widehat{E}_{1}(t,\xi)|\lesssim \lambda_{\xi}^{(k-1)\theta}(1+\lambda_{\xi}^{\theta}t)e^{-\frac{1}{2}\lambda_{\xi}^{\theta}t}\lesssim (1+t)e^{-\delta t},\nonumber
\end{equation}
and consequently, using the relation
\eqref{timederive0}, we obtain
\begin{equation}
    |\partial^{k}_{t}\widehat{E}_{0}(t,\xi)|\lesssim \lambda_{\xi}^{k\theta}(1+\lambda_{\xi}^{\theta}t)e^{-\frac{1}{2}\lambda_{\xi}^{\theta}t}\lesssim (1+t)e^{-\delta t}.\nonumber
\end{equation}
Combining the above estimate with the representation \eqref{timedersoln:fourier}, we have
\begin{equation}\label{estimate of t derivative case 3 R_3}
    |\partial_{t}^{k}\widehat{u}(t,\xi)|\lesssim (1+t)e^{-\delta t}\left(|\widehat{u}_{0}(\xi)| +|\widehat{u}_{1}(\xi)|\right),\quad \text{for } k\geq 1.
\end{equation}
\noindent\textbf{Estimate on $\mathcal{R}_{4}$:} 
From \eqref{case3:e0def} and \eqref{case3:e1def}, we have
\begin{equation}
    |\widehat{E}_{0}(t,\xi)|\lesssim  \left[1+\frac{\lambda_{\xi}^{\theta}}{\sqrt{\lambda_{\xi}^{2\theta}-4\lambda_{\xi}^{\sigma}}}\right]e^{\left(-\frac{1}{2}\lambda_{\xi}^{\theta}+\frac{1}{2}\sqrt{\lambda_{\xi}^{2\theta}-4\lambda_{\xi}^{\sigma}}\right)t}\lesssim e^{-\lambda_{\xi}^{\sigma-\theta}t}\lesssim e^{-\delta t},\nonumber
\end{equation}
and
\begin{equation}
      |\widehat{E}_{1}(t,\xi)|\lesssim  \frac{1}{\sqrt{\lambda_{\xi}^{2\theta}-4\lambda_{\xi}^{\sigma}}}e^{\left(-\frac{1}{2}\lambda_{\xi}^{\theta}+\frac{1}{2}\sqrt{\lambda_{\xi}^{2\theta}-4\lambda_{\xi}^{\sigma}}\right)t}\lesssim e^{-\lambda_{\xi}^{\sigma-\theta}t} \lesssim e^{-\delta t}.\nonumber
\end{equation}
Here we have utilized the relation \eqref{INEQREL1}
to estimate the exponential factor and the fact that $\{\lambda_{\xi}\}_{\xi\in\mathcal{I}}$ is discrete which allow us to bound the denominator factors. Consequently, the representation \eqref{usoln:fourier}  can be estimated as
\begin{equation}\label{estimate of u case4_4}
    |\widehat{u}(t,\xi)|\lesssim e^{-\delta t}\left(|\widehat{u}_{0}(\xi)| +|\widehat{u}_{1}(\xi)|\right).
\end{equation}
Further, for any $\beta>0$, the terms $|\lambda_{\xi}^{\beta}\widehat{E}_{0}(t,\xi)|$ and $|\lambda_{\xi}^{\beta}\widehat{E}_{1}(t,\xi)|$ can be estimated as
\begin{equation}\label{estimate of beta case4_1}
     |\lambda_{\xi}^{\beta}\widehat{E}_{0}(t,\xi)|\lesssim \lambda_{\xi}^{\beta} e^{-\lambda_{\xi}^{\sigma-\theta}t}=\left(\lambda_{\xi}^{\sigma-\theta}t\right)^{\frac{\beta}{\sigma-\theta}} e^{-\lambda_{\xi}^{\sigma-\theta}t}t^{-\frac{\beta}{\sigma-\theta}}
     \lesssim\begin{cases}
         t^{-\frac{\beta}{\sigma-\theta}}&\text{if } t\in (0,1],\\
         e^{-\delta t}&\text{if } t\in [1,\infty),
     \end{cases}
\end{equation}
and
\begin{equation}\label{estimate of beta case4_2}
     |\lambda_{\xi}^{\beta}\widehat{E}_{1}(t,\xi)|\lesssim \lambda_{\xi}^{\beta-\frac{\sigma}{2}} e^{-\lambda_{\xi}^{\sigma-\theta}t}=\left(\lambda_{\xi}^{\sigma-\theta}t\right)^{\frac{\beta-\frac{\sigma}{2}}{\sigma-\theta}} e^{-\lambda_{\xi}^{\sigma-\theta}t}t^{-\frac{\beta-\frac{\sigma}{2}}{\sigma-\theta}}
     \lesssim\begin{cases}
         t^{-\frac{2\beta-\sigma}{2\sigma-2\theta}}&\text{if } t\in (0,1],\\
         e^{-\delta t}&\text{if } t\in [1,\infty),
     \end{cases}
\end{equation}
whenever $2\beta>\sigma$, while
\begin{equation}\label{estimate of beta case4_3}
     |\lambda_{\xi}^{\beta}\widehat{E}_{1}(t,\xi)|\lesssim \lambda_{\xi}^{\beta-\frac{\sigma}{2}} e^{-\lambda_{\xi}^{\sigma-\theta}t}\lesssim e^{-\lambda_{\xi}^{\sigma-\theta}t}
     \lesssim
         e^{-\delta t},
\end{equation}
whenever $2\beta\leq \sigma$.
Also, for any $k \in \mathbb{N}$, the terms $|\partial_{t}^{k}\widehat{E}_{0}(t,\xi)|$ and $|\partial_{t}^{k}\widehat{E}_{1}(t,\xi)|$ can be estimated as
\begin{align}\label{estimate of t case4_1}
    |\partial_{t}^{k}\widehat{E}_{1}(t,\xi)|
    &\lesssim \lambda_{\xi}^{(k-1)\theta}e^{-\lambda_{\xi}^{\sigma-\theta}t}=\left(\lambda_{\xi}^{\sigma-\theta}t\right)^{\frac{(k-1)\theta}{\sigma-\theta}} e^{-\lambda_{\xi}^{\sigma-\theta}t}t^{-\frac{(k-1)\theta}{\sigma-\theta}}\lesssim\begin{cases}
         t^{-\frac{(k-1)\theta}{\sigma-\theta}}&\text{if } t\in (0,1],\\
         e^{-\delta t}&\text{if } t\in [1,\infty),
     \end{cases}
\end{align}
and
\begin{align}\label{estimate of t case4_2}
    |\partial_{t}^{k}\widehat{E}_{0}(t,\xi)|
    &\lesssim \lambda_{\xi}^{k\theta}e^{-\lambda_{\xi}^{\sigma-\theta}t}=\left(\lambda_{\xi}^{\sigma-\theta}t\right)^{\frac{k\theta}{\sigma-\theta}} e^{-\lambda_{\xi}^{\sigma-\theta}t}t^{-\frac{k\theta}{\sigma-\theta}}\lesssim\begin{cases}
         t^{-\frac{k\theta}{\sigma-\theta}}&\text{if } t\in (0,1],\\
         e^{-\delta t}&\text{if } t\in [1,\infty).
     \end{cases}
\end{align}
We are now in position to compute the estimate for $\|u(t)\|_{\mathcal{H}}, \|\mathcal{L}^{\beta}u(t)\|_{\mathcal{H}},$ and $\|\partial_{t}^{k}u(t)\|_{\mathcal{H}}$.

\medskip

\noindent\textbf{Estimate for $\|u(t)\|_{\mathcal{H}}$:} Recalling the Plancherel formula along with the estimates \eqref{R11:usol}, \eqref{estimate of u case4_2}, \eqref{estimate of u case4_3}, and \eqref{estimate of u case4_4}, we have
\begin{multline}
\|u(t)\|_{\mathcal{H}}^{2}\lesssim \sum_{\xi\in \mathcal{R}_{1}}\left(|\widehat{u}_{0}(\xi)|^{2} +t^{2}|\widehat{u}_{1}(\xi)|^{2}\right)+\sum_{\xi\in \mathcal{R}_{2}}e^{-\delta t}\left(|\widehat{u}_{0}(\xi)|^{2} +|\widehat{u}_{1}(\xi)|^{2}\right)\nonumber\\+ \sum_{\xi\in \mathcal{R}_{3}}(1+t)^{2}e^{-\delta t}\left(|\widehat{u}_{0}(\xi)|^{2} +|\widehat{u}_{1}(\xi)|^{2}\right)+\sum_{\xi\in \mathcal{R}_{4}}e^{-\delta t}\left(|\widehat{u}_{0}(\xi)|^{2} +|\widehat{u}_{1}(\xi)|^{2}\right)\lesssim \|u_{0}\|^{2}_{\mathcal{H}}+t^{2}\|u_{1}\|^{2}_{\mathcal{H}}.
\end{multline}
\medskip

\noindent\textbf{Estimate for $\|\mathcal{L}^{\beta}u(t)\|_{\mathcal{H}}$:} Recalling the Sobolev norm along with the estimate
\eqref{R11:spacedersol}, \eqref{estimate of space derivative case3 R_2}, \eqref{estimate of space derivative case3 R_3}, \eqref{estimate of beta case4_1}, \eqref{estimate of beta case4_2} and \eqref{estimate of beta case4_3}, we obtain
\begin{align}\label{case3:spacederuest}
\left\|\mathcal{L}^{\beta}u(t)\right\|_{\mathcal{H}}^{2}
&\lesssim \sum_{\xi\in \mathcal{R}_{2}}e^{-\delta t}\left(|\widehat{u}_{0}(\xi)|^{2} +|\widehat{u}_{1}(\xi)|^{2}\right)
+ \sum_{\xi\in \mathcal{R}_{3}}(1+t)^{2}e^{-\delta t}\left(|\widehat{u}_{0}(\xi)|^{2} +|\widehat{u}_{1}(\xi)|^{2}\right)\nonumber\\
&+\sum_{\xi\in \mathcal{R}_{4}}\begin{cases}
\begin{cases}
        t^{-\frac{2\beta}{\sigma-\theta}}|\widehat{u}_{0}(\xi)|^{2}+t^{-\frac{2\beta-\sigma}{\sigma-\theta}}|\widehat{u}_{1}(\xi)|^{2}&\text{if } t\in (0,1],\\
        e^{-\delta t}\left(|\widehat{u}_{0}(\xi)|^{2}+|\widehat{u}_{1}(\xi)|^{2}\right)&\text{if } t\in [1,\infty),
    \end{cases}&\text{if } 2\beta>\sigma,\\
\begin{cases}
        t^{-\frac{2\beta}{\sigma-\theta}}|\widehat{u}_{0}(\xi)|^{2}+e^{-\delta t}|\widehat{u}_{1}(\xi)|^{2}&\text{if } t\in (0,1],\\
        e^{-\delta t}\left(|\widehat{u}_{0}(\xi)|^{2}+|\widehat{u}_{1}(\xi)|^{2}\right)&\text{if } t\in [1,\infty),
    \end{cases}&\text{if } 2\beta\leq\sigma,
\end{cases} \nonumber\\
&\lesssim \begin{cases}
\begin{cases}
        t^{-\frac{2\beta}{\sigma-\theta}}\|u_{0}\|_{\mathcal{H}}^{2}+t^{-\frac{2\beta-\sigma}{\sigma-\theta}}\|u_{1}\|_{\mathcal{H}}^{2}&\text{if } t\in (0,1],\\
        e^{-\delta t}\left(\|u_{0}\|_{\mathcal{H}}^{2}+\|u_{1}\|_{\mathcal{H}}^{2}\right)&\text{if } t\in [1,\infty),
    \end{cases}&\text{if } 2\beta>\sigma,\\
\begin{cases}
        t^{-\frac{2\beta}{\sigma-\theta}}\|u_{0}\|_{\mathcal{H}}^{2}+e^{-\delta t}\|u_{1}\|_{\mathcal{H}}^{2}&\text{if } t\in (0,1],\\
        e^{-\delta t}\left(\|u_{0}\|_{\mathcal{H}}^{2}+\|u_{1}\|_{\mathcal{H}}^{2}\right)&\text{if } t\in [1,\infty),
    \end{cases}&\text{if } 2\beta\leq\sigma.
\end{cases}
\end{align}

\medskip

\noindent\textbf{Estimate for $\|\partial_{t}^{k}u(t)\|_{\mathcal{H}}$:} 
Combining \eqref{R11:timedersoln}, \eqref{estimate of t derivative case 3 R_2}, \eqref{estimate of t derivative case 3 R_3} \eqref{estimate of t case4_1}, and \eqref{estimate of t case4_2}, we get
\begin{multline}\label{case3:timederuest}
\|\partial_{t}u(t)\|_{\mathcal{H}}^{2}
\lesssim \sum_{\xi\in \mathcal{R}_{1}}
    |\widehat{u}_{1}(\xi)|^{2}+\sum_{\xi\in \mathcal{R}_{2}}e^{-\delta t}\left(|\widehat{u}_{0}(\xi)|^{2} +|\widehat{u}_{1}(\xi)|^{2}\right)+ \sum_{\xi\in \mathcal{R}_{3}}(1+t)^{2}e^{-\delta t}\left(|\widehat{u}_{0}(\xi)|^{2} +|\widehat{u}_{1}(\xi)|^{2}\right)\\+\sum_{\xi\in \mathcal{R}_{4}}\begin{cases}
        t^{-\frac{2\theta}{\sigma-\theta}}|\widehat{u}_{0}(\xi)|^{2}+|\widehat{u}_{1}(\xi)|^{2}&\text{if } t\in (0,1],\\
        e^{-\delta t}|\widehat{u}_{0}(\xi)|^{2}+|\widehat{u}_{1}(\xi)|^{2}&\text{if } t\in [1,\infty),
    \end{cases}\lesssim\begin{cases}
        t^{-\frac{2\theta}{\sigma-\theta}}\|u_{0}\|_{\mathcal{H}}^{2}+\|u_{1}\|_{\mathcal{H}}^{2}&\text{if } t\in (0,1],\\
        e^{-\delta t}\|u_{0}\|_{\mathcal{H}}^{2}+\|u_{1}\|_{\mathcal{H}}^{2}&\text{if } t\in [1,\infty).
    \end{cases}
\end{multline}
Furthermore, for $k\geq 2$, consider 
\begin{multline}\label{case3:timederuestk}
\|\partial_{t}^{k}u(t)\|_{\mathcal{H}}^{2}
\lesssim \sum_{\xi\in \mathcal{R}_{2}}e^{-\delta t}\left(|\widehat{u}_{0}(\xi)|^{2} +|\widehat{u}_{1}(\xi)|^{2}\right)+ \sum_{\xi\in \mathcal{R}_{3}}(1+t)^{2}e^{-\delta t}\left(|\widehat{u}_{0}(\xi)|^{2} +|\widehat{u}_{1}(\xi)|^{2}\right)\\+\sum_{\xi\in \mathcal{R}_{4}} \begin{cases}
 t^{-\frac{2k\theta}{\sigma-\theta}}|\widehat{u}_{0}(\xi)|^{2}+t^{-\frac{2(k-1)\theta}{\sigma-\theta}}|\widehat{u}_{1}(\xi)|^{2}&\text{if } t\in (0,1],\\
   e^{-\delta t}\left(|\widehat{u}_{0}(\xi)|^{2}+|\widehat{u}_{1}(\xi)|^{2}\right)&\text{if } t\in [1,\infty),
\end{cases}\lesssim
\begin{cases}
 t^{-\frac{2k\theta}{\sigma-\theta}}\|u_{0}\|_{\mathcal{H}}^{2}+t^{-\frac{2(k-1)\theta}{\sigma-\theta}}\|u_{1}\|_{\mathcal{H}}^{2}&\text{if } t\in (0,1],\\
   e^{-\delta t}\left(\|u_{0}\|_{\mathcal{H}}^{2}+\|u_{1}\|_{\mathcal{H}}^{2}\right)&\text{if } t\in [1,\infty).
\end{cases}
\end{multline}
This completes the proof of Theorem \ref{thm:case3}.
\end{proof}

Upon establishing the decay estimates for the  non-effective damping, we may now prove Theorem \ref{thm:case0old} for the case of non-effective damping, with certain adjustments outlined in the subsequent remark.

\begin{remark}\label{rmk:noneff-damping}
The estimates \eqref{estimate of beta case4_1}, \eqref{estimate of beta case4_2} and \eqref{estimate of beta case4_3} on $\mathcal{R}_{4}$ can also be estimated as
\begin{align}
     |\lambda_{\xi}^{\beta}\widehat{E}_{0}(t,\xi)|\lesssim \lambda_{\xi}^{\beta}e^{-\delta t}\quad \text{and}\quad 
     |\lambda_{\xi}^{\beta}\widehat{E}_{1}(t,\xi)|\lesssim  e^{-\delta t}\frac{\lambda_{\xi}^{\beta}}{\sqrt{\lambda_{\xi}^{2\theta}-4\lambda_{\xi}^{\sigma}}}\lesssim e^{-\delta t}\lambda_{\xi}^{\beta-\frac{\sigma}{2}}.\nonumber
\end{align}
Consequently, we can estimate
\begin{equation}
|\lambda_{\xi}^{\beta}\widehat{u}(t,\xi)|\lesssim e^{-\delta t}
\left(\lambda_{\xi}^{\beta}|\widehat{u}_{0}(\xi)|+\lambda_{\xi}^{\beta-\frac{\sigma}{2}}|\widehat{u}_{1}(\xi)|\right).\nonumber
\end{equation}
Hence, the estimate \eqref{case3:spacederuest}  will be replaced by \eqref{thm5:spaceder}.
Similarly, the estimate \eqref{estimate of t case4_1} and \eqref{estimate of t case4_2} can also be estimated as
      \begin{align}
    |\partial_{t}^{k}\widehat{E}_{1}(t,\xi)|\lesssim \lambda_{\xi}^{(k-1)\theta}e^{-\delta t}\quad\text{and}\quad |\partial_{t}^{k}\widehat{E}_{0}(t,\xi)|\lesssim  \lambda_{\xi}^{k\theta}e^{-\delta t}.\nonumber
\end{align}
So, we obtain
\begin{equation}
      |\partial_{t}^{k}\widehat{u}(t,\xi)|\lesssim e^{-\delta t}\left(\lambda_{\xi}^{k\theta}|\widehat{u}_{0}(\xi)| +\lambda_{\xi}^{(k-1)\theta}|\widehat{u}_{1}(\xi)|\right),\quad \text{for } k\geq 1,\nonumber
\end{equation}
and hence the estimate \eqref{case3:timederuest} and \eqref{case3:timederuestk} will be replaced by \eqref{thm5:timeder}. This completes the proof of Theorem \ref{thm:case0old} for the case $2\theta\in(\sigma,2\sigma)$.
\end{remark}
The proof of Theorem \ref{thm:case0old} for effective damping, critical damping and non-effective damping is now complete, as demonstrated in Remarks \ref{rmk:eff-damping}, \ref{rmk:critical-damping}, and \ref{rmk:noneff-damping}. The estimates in the high-frequency region are alternatively derived in these remarks to obtain exponential decay in instead of polynomial decay. Nevertheless, there are certain terms $\mathcal{R}_{1}$ in $\|u\|_{\mathcal{H}}$ and $\|\partial_{t}u\|_{\mathcal{H}}$ that do not show any decay.

Let us now proceed to prove the decay estimates for viscoelastic type damping in Theorem \ref{thm:case0old}.
\begin{proof}
Similar to the previous cases, it is reasonable to consider the following partition of $\mathcal{I}$ in this case:
 
    \begin{enumerate}
    \item $\mathcal{R}_{1}=\{\xi\in \mathcal{I}: \lambda_{\xi}^{\theta}=0\},$
    \item  $\mathcal{R}_{2}=\left\{\xi\in \mathcal{I}: 0<\lambda_{\xi}^{\theta}<4\right\}$,
    \item $\mathcal{R}_{3}=\left\{\xi\in \mathcal{I}: \lambda_{\xi}^{\theta}=4\right\}$, and
    \item $ \mathcal{R}_{4}=\left\{\xi\in \mathcal{I}: \lambda_{\xi}^{\theta}>4\right\}$.
\end{enumerate}
Combining the relation \eqref{case3:root new} with representation \eqref{e0e1formula}, we get
\begin{equation}\label{case3:e0def new}
    \widehat{E}_{0}(t,\xi)=\begin{cases}1 &\text{if } \xi\in\mathcal{R}_{1},\\
    \left[\cos \left(\frac{1}{2}\sqrt{4\lambda_{\xi}^{\theta}-\lambda_{\xi}^{2\theta}} t\right)+\frac{\lambda_{\xi}^{\theta}\sin\left(\frac{1}{2}\sqrt{4\lambda_{\xi}^{\theta}-\lambda_{\xi}^{2\theta}} t\right)}{\sqrt{4\lambda_{\xi}^{\theta}-\lambda_{\xi}^{2\theta}}}  \right]e^{-\frac{1}{2}\lambda_{\xi}^{\theta}t}& \text {if } \xi\in\mathcal{R}_{2},\\
   \left[1+\frac{1}{2}\lambda_{\xi}^{\theta}t\right]e^{-\frac{1}{2}\lambda_{\xi}^{\theta}t}& \text {if } \xi\in\mathcal{R}_{3}, \\
   \left[
    \cosh \left(\frac{1}{2}\sqrt{\lambda_{\xi}^{2\theta}-4\lambda_{\xi}^{\theta}} t\right)+\frac{\lambda_{\xi}^{\theta}\sinh \left(\frac{1}{2}\sqrt{\lambda_{\xi}^{2\theta}-4\lambda_{\xi}^{\theta}} t\right)}{\sqrt{\lambda_{\xi}^{2\theta}-4\lambda_{\xi}^{\theta}}}\right]e^{-\frac{1}{2}\lambda_{\xi}^{\theta}t} & \text {if }  \xi\in\mathcal{R}_{4},\end{cases}
\end{equation}
and
\begin{equation}\label{case3:e1def new}
    \widehat{E}_{1}(t,\xi)=\begin{cases}t &\text{if } \xi\in\mathcal{R}_{1},\\
    \frac{\sin \left(\frac{1}{2}\sqrt{4\lambda_{\xi}^{\theta}-\lambda_{\xi}^{2\theta}} t\right)}{\sqrt{4\lambda_{\xi}^{\theta}-\lambda_{\xi}^{2\theta}}} e^{-\frac{1}{2}\lambda_{\xi}^{\theta}t}& \text {if } \xi\in\mathcal{R}_{2},\\ te^{-\frac{1}{2}\lambda_{\xi}^{\theta}t} & \text {if } \xi\in\mathcal{R}_{3},\\ \frac{\sinh \left(\frac{1}{2}\sqrt{\lambda_{\xi}^{2\theta}-4\lambda_{\xi}^{\theta}} t\right)}{\sqrt{\lambda_{\xi}^{2\theta}-4\lambda_{\xi}^{\theta}}}e^{-\frac{1}{2}\lambda_{\xi}^{\theta}t} & \text {if } \xi\in\mathcal{R}_{4}.\end{cases}
\end{equation}
Now we will estimate $\widehat{E}_{0}(t,\xi)$ and $\widehat{E}_{1}(t,\xi)$ on $\mathcal{R}_{j}$ for $j=1,2,3,4$.

\medskip

\noindent\textbf{Estimate on $\mathcal{R}_{1}$:} Since $\widehat{E}_{0}(t,\xi)$   and $\widehat{E}_{1}(t,\xi)$ have the same expression as in the effective damping case on $\mathcal{R}_{1}$, therefore the estimates for $\widehat{u}(t,\xi)$, $\lambda_{\xi}^{\beta}\widehat{u}(t,\xi)$ and $\partial_{t}^{k}\widehat{u}(t,\xi)$ in this case are exactly equal to \eqref{R11:usol}, \eqref{R11:spacedersol}, and \eqref{R11:timedersoln}, respectively.

\noindent\textbf{Estimate on $\mathcal{R}_{2}$:} 
From \eqref{case3:e0def new} and \eqref{case3:e1def new}, using the fact that $\frac{\lambda_{\xi}^{\theta}}{\sqrt{4\lambda_{\xi}^{\theta}-\lambda_{\xi}^{2\theta}}}$ is bounded and $\left|\sin x\right|\leq 1$, we get
\begin{align}\label{estimate of u case4_2 new}
     |\widehat{E}_{0}(t,\xi)|,~|\widehat{E}_{1}(t,\xi)|\lesssim e^{-\delta t}\quad \text{and hence we get }|\widehat{u}(t,\xi)|\lesssim e^{-\delta t}\left(|\widehat{u}_{0}(\xi)| +|\widehat{u}_{1}(\xi)|\right).
     \end{align}
Further, for any $\beta>0$ we have
\begin{align}
     |\lambda_{\xi}^{\beta}\widehat{E}_{0}(t,\xi)|\lesssim \lambda_{\xi}^{\beta}e^{-\frac{1}{2}\lambda_{\xi}^{\theta}t}\lesssim e^{-\delta t}\quad \text{and}\quad
|\lambda_{\xi}^{\beta}\widehat{E}_{1}(t,\xi)|\lesssim e^{-\delta t}.\nonumber
\end{align}
Consequently, we can estimate
\begin{equation}\label{estimate of space derivative case3 R_2 new}
|\lambda_{\xi}^{\beta}\widehat{u}(t,\xi)|\lesssim e^{-\delta t}
\left(|\widehat{u}_{0}(\xi)|+|\widehat{u}_{1}(\xi)|\right).
\end{equation}
Recalling the relation \eqref{timederive1} for $\widehat{E}_{1}(t,\xi)$ on $\mathcal{R}_{3}$, we have
\begin{multline}
    \partial_{t}^{k}\widehat{E}_{1}(t,\xi)=\\\frac{\left(-\frac{1}{2}\lambda_{\xi}^{\theta}+\frac{i}{2}\sqrt{4\lambda_{\xi}^{\theta}-\lambda_{\xi}^{2\theta}}\right)^{k}}{i\sqrt{4\lambda_{\xi}^{\theta}-\lambda_{\xi}^{2\theta}}}e^{\left(-\frac{1}{2}\lambda_{\xi}^{\theta}+\frac{i}{2}\sqrt{4\lambda_{\xi}^{\theta}-\lambda_{\xi}^{2\theta}}\right)t}-\frac{\left(-\frac{1}{2}\lambda_{\xi}^{\theta}-\frac{i}{2}\sqrt{4\lambda_{\xi}^{\theta}-\lambda_{\xi}^{2\theta}}\right)^{k}}{i\sqrt{4\lambda_{\xi}^{\theta}-\lambda_{\xi}^{2\theta}}}e^{\left(-\frac{1}{2}\lambda_{\xi}^{\theta}-\frac{i}{2}\sqrt{4\lambda_{\xi}^{\theta}-\lambda_{\xi}^{2\theta}}\right)t}.\nonumber
\end{multline}
Using the boundedness of $\frac{1}{\sqrt{4\lambda_{\xi}^{\theta}-\lambda_{\xi}^{2\theta}}}$, we obtain
\begin{align*}
    |\partial_{t}^{k}\widehat{E}_{1}(t,\xi)|&\lesssim\left(\left|\left(-\frac{1}{2}\lambda_{\xi}^{\theta}+\frac{i}{2}\sqrt{4\lambda_{\xi}^{\theta}-\lambda_{\xi}^{2\theta}}\right)^{k}\right|+\left|\left(-\frac{1}{2}\lambda_{\xi}^{\theta}-\frac{i}{2}\sqrt{4\lambda_{\xi}^{\theta}-\lambda_{\xi}^{2\theta}}\right)^{k}\right|\right)e^{-\frac{1}{2}\lambda_{\xi}^{\theta}t} \lesssim e^{-\delta t},
\end{align*}
 Consequently using \eqref{timederive0}, we get
\begin{equation}
    |\partial_{t}^{k}\widehat{E}_{0}(t,\xi)|\lesssim \lambda_{\xi}^{\theta}|\partial_{t}^{k-1}\widehat{E}_{1}(t,\xi)|\lesssim \lambda_{\xi}^{\frac{k+1}{2}\theta}e^{-\frac{1}{2}\lambda_{\xi}^{\theta}t}\lesssim e^{-\delta t}.\nonumber
\end{equation}
Hence we obtain
\begin{equation}\label{estimate of t derivative case 3 R_2 new}
      |\partial_{t}^{k}\widehat{u}(t,\xi)|\lesssim e^{-\delta t}\left(|\widehat{u}_{0}(\xi)| +|\widehat{u}_{1}(\xi)|\right),\quad \text{for } k\geq 1.
\end{equation}

\noindent\textbf{Estimate on $\mathcal{R}_{3}$:} From \eqref{case3:e0def new} and \eqref{case3:e1def new}, we have
\begin{equation}
     |\widehat{E}_{0}(t,\xi)|\lesssim (1+t)e^{-\delta t}\quad\text{and}\quad     |\widehat{E}_{1}(t,\xi)|\lesssim te^{-\delta t}.\nonumber
\end{equation}
This gives
\begin{equation}\label{estimate of u case4_3 new}
    |\widehat{u}(t,\xi)|\lesssim (1+t)e^{-\delta t}|\widehat{u}_{0}(\xi)| +te^{-\delta t}\widehat{u}_{1}(\xi)|.
\end{equation}
Consequently, for any $\beta>0$, 
\begin{equation}\label{estimate of space derivative case3 R_3 new}
    \lambda_{\xi}^{\beta}|\widehat{u}(t,\xi)|\lesssim (1+t)e^{-\delta t}|\widehat{u}_{0}(\xi)| +te^{-\delta t}|\widehat{u}_{1}(\xi)|.
\end{equation}
Furthermore, using the relation \eqref{timederive1} and \eqref{case3:e1def new}, we have
\begin{equation}
    |\partial^{k}_{t}\widehat{E}_{1}(t,\xi)|\lesssim \lambda_{\xi}^{(k-1)\theta}(1+\lambda_{\xi}^{\theta}t)e^{-\frac{1}{2}\lambda_{\xi}^{\theta}t}\lesssim (1+t)e^{-\delta t},\nonumber
\end{equation}
and consequently, using the relation
\eqref{timederive0}, we obtain
\begin{equation}
    |\partial^{k}_{t}\widehat{E}_{0}(t,\xi)|\lesssim \lambda_{\xi}^{k\theta}(1+\lambda_{\xi}^{\theta}t)e^{-\frac{1}{2}\lambda_{\xi}^{\theta}t}\lesssim (1+t)e^{-\delta t}.\nonumber
\end{equation}
Combining the above estimate with the representation \eqref{timedersoln:fourier}, we have
\begin{equation}\label{estimate of t derivative case 3 R_3 new}
    |\partial_{t}^{k}\widehat{u}(t,\xi)|\lesssim (1+t)e^{-\delta t}\left(|\widehat{u}_{0}(\xi)| +|\widehat{u}_{1}(\xi)|\right),\quad \text{for } k\geq 1.
\end{equation}
\noindent\textbf{Estimate on $\mathcal{R}_{4}$:} 
From \eqref{case3:e0def new} and \eqref{case3:e1def new}, we have
\begin{equation}
    |\widehat{E}_{0}(t,\xi)|\lesssim  \left[1+\frac{\lambda_{\xi}^{\theta}}{\sqrt{\lambda_{\xi}^{2\theta}-4\lambda_{\xi}^{\theta}}}\right]e^{\left(-\frac{1}{2}\lambda_{\xi}^{\theta}+\frac{1}{2}\sqrt{\lambda_{\xi}^{2\theta}-4\lambda_{\xi}^{\theta}}\right)t}\lesssim e^{-\lambda_{\xi}^{\theta-\theta}t}\lesssim e^{-t},\nonumber
\end{equation}
and
\begin{equation}
      |\widehat{E}_{1}(t,\xi)|\lesssim  \frac{1}{\sqrt{\lambda_{\xi}^{2\theta}-4\lambda_{\xi}^{\theta}}}e^{\left(-\frac{1}{2}\lambda_{\xi}^{\theta}+\frac{1}{2}\sqrt{4\lambda_{\xi}^{\theta}-\lambda_{\xi}^{2\theta}}\right)t}\lesssim e^{-\lambda_{\xi}^{\theta-\theta}t} \lesssim e^{-t}.\nonumber
\end{equation}
Here we have utilized the relation \eqref{INEQREL1}
to estimate the exponential factor and the fact that $\{\lambda_{\xi}\}_{\xi\in\mathcal{I}}$ is discrete which allow us to bound the denominator factors. Consequently, the representation \eqref{usoln:fourier}  can be estimated as
\begin{equation}\label{estimate of u case4_4 new}
    |\widehat{u}(t,\xi)|\lesssim e^{-t} \left(|\widehat{u}_{0}(\xi)| +|\widehat{u}_{1}(\xi)|\right).
\end{equation}
Further, for any $\beta>0$,
\begin{equation}\label{estimate of beta case4_1 new}
     |\lambda_{\xi}^{\beta}\widehat{E}_{0}(t,\xi)|\lesssim \lambda_{\xi}^{\beta} e^{-t}
\end{equation}
and
\begin{equation}\label{estimate of beta case4_2 new}
     |\lambda_{\xi}^{\beta}\widehat{E}_{1}(t,\xi)|\lesssim \lambda_{\xi}^{\beta-\frac{\theta}{2}} e^{-t},
\end{equation}
whenever $2\beta>\theta$, while
\begin{equation}\label{estimate of beta case4_3 new}
     |\lambda_{\xi}^{\beta}\widehat{E}_{1}(t,\xi)|\lesssim \lambda_{\xi}^{\beta-\frac{\theta}{2}} e^{-t}\lesssim e^{-t},
\end{equation}
whenever $2\beta\leq \theta$.
Also, for any $k \in \mathbb{N}$, the terms $|\partial_{t}^{k}\widehat{E}_{0}(t,\xi)|$ and $|\partial_{t}^{k}\widehat{E}_{1}(t,\xi)|$ can be estimated as
\begin{align}\label{estimate of t case4_1 new}
    |\partial_{t}^{k}\widehat{E}_{1}(t,\xi)|
    &\lesssim \lambda_{\xi}^{(k-1)\theta} e^{-t},
\end{align}
and
\begin{align}\label{estimate of t case4_2 new}
    |\partial_{t}^{k}\widehat{E}_{0}(t,\xi)|
    &\lesssim \lambda_{\xi}^{k\theta}e^{-t} .
\end{align}
We are now in position to compute the estimate for $\|u(t)\|_{\mathcal{H}}, \|\mathcal{L}^{\beta}u(t)\|_{\mathcal{H}},$ and $\|\partial_{t}^{k}u(t)\|_{\mathcal{H}}$.

\medskip

\noindent\textbf{Estimate for $\|u(t)\|_{\mathcal{H}}$:} Recalling the Plancherel formula along with the estimates \eqref{R11:usol}, \eqref{estimate of u case4_2 new}, \eqref{estimate of u case4_3 new}, and \eqref{estimate of u case4_4 new}, we have
\begin{multline}
\|u(t)\|_{\mathcal{H}}^{2}\lesssim \sum_{\xi\in \mathcal{R}_{1}}\left(|\widehat{u}_{0}(\xi)|^{2} +t^{2}|\widehat{u}_{1}(\xi)|^{2}\right)+\sum_{\xi\in \mathcal{R}_{2}}e^{-\delta t}\left(|\widehat{u}_{0}(\xi)|^{2} +|\widehat{u}_{1}(\xi)|^{2}\right)\nonumber\\+ \sum_{\xi\in \mathcal{R}_{3}}(1+t)^{2}e^{-\delta t}\left(|\widehat{u}_{0}(\xi)|^{2} +|\widehat{u}_{1}(\xi)|^{2}\right)+\sum_{\xi\in \mathcal{R}_{4}} e^{-2t}\left(|\widehat{u}_{0}(\xi)|^{2} +|\widehat{u}_{1}(\xi)|^{2}\right)\lesssim \|u_{0}\|^{2}_{\mathcal{H}}+t^{2}\|u_{1}\|^{2}_{\mathcal{H}}.
\end{multline}
\medskip

\noindent\textbf{Estimate for $\|\mathcal{L}^{\beta}u(t)\|_{\mathcal{H}}$:} Recalling the Sobolev norm along with the estimate
\eqref{R11:spacedersol}, \eqref{estimate of space derivative case3 R_2 new}, \eqref{estimate of space derivative case3 R_3 new}, \eqref{estimate of beta case4_1 new}, \eqref{estimate of beta case4_2 new} and \eqref{estimate of beta case4_3 new}, we obtain
\begin{align}\label{case3:spacederuest new}
\left\|\mathcal{L}^{\beta}u(t)\right\|_{\mathcal{H}}^{2}
&\lesssim \sum_{\xi\in \mathcal{R}_{2}}e^{-\delta t}\left(|\widehat{u}_{0}(\xi)|^{2} +|\widehat{u}_{1}(\xi)|^{2}\right)
+ \sum_{\xi\in \mathcal{R}_{3}}(1+t)^{2}e^{-\delta t}\left(|\widehat{u}_{0}(\xi)|^{2} +|\widehat{u}_{1}(\xi)|^{2}\right)\nonumber\\
&+\sum_{\xi\in \mathcal{R}_{4}} e^{-2t} \left(\lambda_{\xi}^{2\beta}|\widehat{u}_{0}(\xi)|^{2}+\lambda_{\xi}^{2\beta-\theta}|\widehat{u}_{1}(\xi)|^{2}\right) \lesssim e^{-\delta t} \begin{cases}
\left(\|u_{0}\|^{2}_{\mathcal{H}^{2\beta}_{\mathcal{L}}}+\|u_{1}\|^{2}_{\mathcal{H}^{2\beta-\theta}_{\mathcal{L}}}\right)&\text{if } 2\beta>\theta,\\
\left(\|u_{0}\|^{2}_{\mathcal{H}^{2\beta}_{\mathcal{L}}}+\|u_{1}\|^{2}_{\mathcal{H}}\right)&\text{if } 2\beta\leq\theta.
\end{cases}
\end{align}

\medskip

\noindent\textbf{Estimate for $\|\partial_{t}^{k}u(t)\|_{\mathcal{H}}$:} 
Combining \eqref{R11:timedersoln}, \eqref{estimate of t derivative case 3 R_2 new}, \eqref{estimate of t derivative case 3 R_3 new}, \eqref{estimate of t case4_1 new}, and \eqref{estimate of t case4_2 new}, we get
\begin{align}\label{case3:timederuest new}
\|\partial_{t}u(t)\|_{\mathcal{H}}^{2}
&\lesssim \sum_{\xi\in \mathcal{R}_{1}}
    |\widehat{u}_{1}(\xi)|^{2}+\sum_{\xi\in \mathcal{R}_{2}}e^{-\delta t}\left(|\widehat{u}_{0}(\xi)|^{2} +|\widehat{u}_{1}(\xi)|^{2}\right)+ \sum_{\xi\in \mathcal{R}_{3}}(1+t)^{2}e^{-\delta t}\left(|\widehat{u}_{0}(\xi)|^{2} +|\widehat{u}_{1}(\xi)|^{2}\right)\nonumber\\&+\sum_{\xi\in \mathcal{R}_{4}} e^{-2t}
\left(\lambda_{\xi}^{2\theta}|\widehat{u}_{0}(\xi)|^{2} +|\widehat{u}_{1}(\xi)|^{2}\right)\lesssim e^{-\delta t}
\|u_{0}\|^{2}_{\mathcal{H}^{2\theta}_{\mathcal{L}}}+\|u_{1}\|^{2}_{\mathcal{H}}.
\end{align}
Furthermore, for $k\geq 2$, we have 
\begin{align}\label{case3:timederuestk new}
\|\partial_{t}^{k}u(t)\|_{\mathcal{H}}^{2}
&\lesssim \sum_{\xi\in \mathcal{R}_{2}}e^{-\delta t}\left(|\widehat{u}_{0}(\xi)|^{2} +|\widehat{u}_{1}(\xi)|^{2}\right)+ \sum_{\xi\in \mathcal{R}_{3}}(1+t)^{2}e^{-\delta t}\left(|\widehat{u}_{0}(\xi)|^{2} +|\widehat{u}_{1}(\xi)|^{2}\right)\nonumber\\
&+\sum_{\xi\in \mathcal{R}_{4}} e^{-2t}
\left(\lambda_{\xi}^{2k\theta}|\widehat{u}_{0}(\xi)|^{2} +\lambda_{\xi}^{(2k-1)\theta}|\widehat{u}_{1}(\xi)|^{2}\right)\lesssim e^{-\delta t}
\left(\|u_{0}\|^{2}_{\mathcal{H}^{2k\theta}_{\mathcal{L}}}+\|u_{1}\|^{2}_{\mathcal{H}^{2(k-1)\theta}_{\mathcal{L}}}\right).
\end{align}
\end{proof}

The proof of Theorem \ref{thm:expdecay} can now be instantly illustrated due to the development of decay estimates with additional Sobolev regularity. 
\begin{proof}[Proof of Theorem \ref{thm:expdecay}]
Given that the operator $\mathcal{L}$ is strictly positive, the set $\mathcal{R}_{1}=\{\xi\in\mathcal{I}:\lambda_{\xi}=0\}$ is empty. Consequently, the low-frequency region $\mathcal{R}_{1}$ does not contribute to $\|u\|_{\mathcal{H}}$ and $\|\partial_{t}u\|_{\mathcal{H}}$, while in other regions i.e. $\mathcal{R}_{j}$ for $j=2,3,4$, we already have exponential decay. The proof of Theorem \ref{thm:expdecay} is now complete. 
\end{proof}
\section{Application}\label{section-7}
Having established the decay estimates for both the undamped and damped cases of the semilinear Cauchy problem \eqref{main:Hilbert}, we are now able to employ these results to the question of global/local (in-time) well-posedness. This application emphasizes the decay framework's efficiency in managing nonlinear effects and guaranteeing the long-term existence and uniqueness of solutions under appropriate assumptions regarding the initial data and the nonlinearity. 

We can recognize our Hilbert space $\mathcal{H}$ with $L^{2}(\Omega)$  for some measure space $\Omega$, thereby enabling us to utilize the entire scale of $L^{p}(\Omega)$ spaces. An exemplary case of interest is $\mathcal{H}=L^{2}(\mathbb{R}^{n})$ or $\mathcal{H}=L^{2}(\mathcal{M})$  for a compact manifold $\mathcal{M}$. In this notation, we denote $\|\cdot\|_{\mathcal{H}}=\|\cdot\|_{2}$ and $\|\cdot\|_{p}=\|\cdot\|_{L^{p}(\Omega)}$.
\begin{defn}[Gagliardo–Nirenberg index]\label{gagliardo} For the operator $\mathcal{L}$, we say that $p\geq 1$ is Gagliardo Nirenberg admissible if the following Gagliardo–Nirenberg type inequality
\begin{eqnarray}\label{gagliardo ineq}
    \|u\|_{2p}\lesssim \|\mathcal{L}^{\sigma/2}u\|^{\alpha}_{2}\|u\|_{2}^{1-\alpha} 
\end{eqnarray}
holds for some $\alpha=\alpha(p)\in [0,1]$.
\end{defn}
\begin{remark}
Here are some examples of     Gagliardo–Nirenberg admissible indices:
\begin{itemize}
    \item \textbf{For harmonic sscillator $\mathcal{L}=-\Delta+|x|^2$ on $L^{2}(\mathbb{R}^{n})$ or for Laplacian on $L^{2}(\mathbb{S}^{n})$:} The Gagliardo–Nirenberg admissible indices are given by
    \begin{equation}
        \begin{cases}
            n=1\text{ and } n=2 &:~~1\leq p<\infty,\\
            n\geq 3&:~~1\leq p\leq \frac{n}{n-2\sigma}.
        \end{cases}
    \end{equation}
    \item \textbf{Laplacian on compact Riemannian manifolds:} The Gagliardo–Nirenberg admissible indices are given by
    \begin{equation}
        \begin{cases}
             n=2 &:~~1\leq p<\infty,\\
            n\geq 3&:~~1\leq p\leq \frac{n}{n-2\sigma}.
        \end{cases}
    \end{equation}
      \item \textbf{Twisted Laplacian (Landau Hamiltonian) on $\mathbb{C}^{n}$ :} The Gagliardo–Nirenberg admissible indices are given by
    \begin{equation}
        \begin{cases}
             n=1 &:~~1\leq p<\infty,\\
            n\geq 2&:~~1\leq p\leq \frac{n}{n-\sigma}.
        \end{cases}
    \end{equation}
\end{itemize}
\end{remark}
For more details regarding the examples in the above remark, one can refer to \cite[Section 3]{mic-wave2} and the references therein.
 \begin{theorem}[Global Existence]\label{global:existence:l2}
	Let $\sigma>0$,  $\theta\in[0,\sigma]$ and $\mathcal{L}$ be strictly positive operator. Let $p\geq 1$ be Gagliardo-Nirenberg admissible according to Definition \ref{gagliardo} and $f$ satisfies the following relation:
 \begin{equation}\label{fuj:condition}
     |f(u)-f(v)|\lesssim(|u|^{p-1}+|v|^{p-1})|u-v|\quad \text{and }\quad f(0)=0.
 \end{equation}
 Also assume that
 \begin{equation}
     (u_{0},u_{1})\in \mathcal{A}^{\sigma,\theta}:=\begin{cases}
         \mathcal{H}_{\mathcal{L}}^{\sigma}\times \mathcal{H}\quad &\mathrm{for }~~2\theta\in[0,\sigma]\\
         \mathcal{H}_{\mathcal{L}}^{2\theta}\times \mathcal{H}\quad &\mathrm{for }~~2\theta\in(\sigma,2\sigma].\nonumber
     \end{cases}
 \end{equation}
	Then, there exists a sufficiently small $\varepsilon>0$ such that for any initial data
	\begin{equation*}
		(u_{0},u_{1})\in \mathcal{A}^{\sigma,\theta} ~~\text{ satisfying }  ~~		\|(u_{0},u_{1})\|_{\mathcal{A}^{\sigma,\theta}}<\epsilon,
	\end{equation*}
	there is a unique global solution $u \in C(\mathbb{R}_{+}; \mathcal{H}_{\mathcal{L}}^{\sigma})\cap C^{1}(\mathbb{R}_{+}; \mathcal{H})$ to the Cauchy problem \eqref{main:Hilbert}.
    Moreover, $u$ also satisfies the estimates given in Theorem \ref{thm:expdecay} for $(k,\beta)\in\{(0,0),(0,\sigma/2),(1,0)\}$ and for all $t> 0$.
\end{theorem}
As mentioned in the prior discussion, we denote $\|\cdot\|_{\mathcal{H}}=\|\cdot\|_{2}$.
\begin{proof} The expression \begin{equation}
u^{\mathrm{lin}}(t)=E_{0}*_{\mathcal{L}}u_{0}(t)+E_{1}*_{\mathcal{L}}u_{1}(t),\nonumber
\end{equation}
represents the solution of the homogeneous Cauchy problem \eqref{main:linear:Hilbert},
	where $*_{\mathcal{L}}$ denotes  $\mathcal{L}$-convolution. The solution to the semilinear Cauchy problem \eqref{main:Hilbert}, using Duhamel's principle, can be expressed as
	\begin{equation*}
		u(t)=u^{\operatorname{lin}}(t)+\int_{0}^{t}E_{1}(t-\tau)*_{\mathcal{L}}f(u)(\tau)\mathrm{d}\tau=u^{\operatorname{lin}}(t)+u^{\operatorname{non}}(t).
	\end{equation*}
    Now, we introduce the solution space $X:=C(\mathbb{R}_{+}; \mathcal{H}_{\mathcal{L}}^{\sigma})\cap C^{1}(\mathbb{R}_{+}; \mathcal{H})$ 
associated with the norm
 \begin{equation*}
\left\|u\right\|_{X}:=\sup\limits_{t\in[0,\infty)}\left\{e^{\delta t}\left(\left\|u(t)\right\|_{2}+\left\|\mathcal{L}^{\sigma/2}u(t)\right\|_{2}+\left\|\partial_{t}u(t)\right\|_{2}\right)\right\}.
\end{equation*}
Let $\mathcal{N}:X\to X$ be the operator defined as follows:
\begin{equation}\label{nudef}
\mathcal{N} u(t):=u^{\operatorname{lin}}(t)+u^{\operatorname{non}}(t),\quad \text{for all }t>0.
\end{equation}
We will illustrate that the following inequalities are satisfied by the operator $\mathcal{N}$:
\begin{eqnarray}
 \|\mathcal{N} u\|_{X}&\leq& L\label{ineq1}, \quad\text{for some suitable } L>0,\\
 \|\mathcal{N} u-\mathcal{N} v\|_{X} &\leq& \frac{1}{r}\|u-v\|_{X},\quad\text{for some suitable } r>1,\label{inequali2}
\end{eqnarray}
for all $u, v \in X$.
We will establish that $\mathcal{N}$ is a contraction mapping on the Banach space  $X$ using the inequalities \eqref{ineq1} and \eqref{inequali2}. As a result, the Banach fixed-point theorem asserts that $\mathcal{N}$ admits a unique fixed point in $X$, which corresponds to a unique global solution $u$  of the equation $\mathcal{N}u=u$.
We now proceed to prove the inequalities \eqref{ineq1} and \eqref{inequali2}.

\medskip

\noindent\textbf{Estimate for $\|\mathcal{N}u\|_{X}$:}
Employing the definition of the norm   $\|\cdot\|_{X}$ and the estimates provide in Theorem \ref{thm:expdecay}, we get the following:
\begin{equation}\label{fineqn0}
    \|u^{\operatorname{lin}}(t)\|_{X}\lesssim \|(u_{0},u_{1})\|_{\mathcal{A}^{\sigma,\theta}},
\end{equation}
and hence $u^{\operatorname{lin}}\in X$. To determine the norm $\left\|u^{\operatorname{non}}\right\|_{X}$,
  the following component norms must be estimated:  
\begin{equation*}
	\|\partial_{t}^{i}\mathcal{L}^{j}u^{\operatorname{non}}(t)\|_{2},\quad (i,j)\in\{(0,0),(0,\sigma/2),(1,0)\}.
\end{equation*}
Using the assumption \eqref{fuj:condition} on $f$, we get
$$
\|(f(u)-f(v))(t, \cdot)\|_2^2 \lesssim \int_{\Omega} \left(|u(t,\cdot)|^{p-1}+|v(t,\cdot)|^{p-1}\right)^2|u(t,\cdot)-v(t,\cdot)|^2.
$$
Consequently, by the H\"{o}lder inequality, we get
$$
\|(f(u)-f(v))(t, \cdot)\|_2^2 \lesssim \left(\|u(t, \cdot)\|_{2 p}^{p-1}+\|v(t, \cdot)\|_{2 p}^{p-1}\right)^2\|(u-v)(t, \cdot)\|_{2 p}^2,
$$
since $\frac{1}{\frac{p}{p-1}}+\frac{1}{p}=1$. Now, by using the Gagliardo-Nirenberg inequality \eqref{gagliardo ineq} and the Young's inequality
$$x^{\alpha}y^{1-\alpha}\leq \alpha x+(1-\alpha)y,\quad \text{for } 0\leq \alpha\leq 1\quad \text{and }x,y\geq0,$$
 we obtain
\begin{align}\label{estimate of f}
\|(f(u) -f(v))(t, \cdot) \|_2 &\lesssim \left[\left(\|\mathcal{L}^{\sigma/2}u(t,\cdot)\|^{\alpha}_{2}\|u\|_{2}^{1-\alpha} \right)^{p-1}+\left(\|\mathcal{L}^{\sigma/2}v(t,\cdot)\|^{\alpha}_{2}\|v(t,\cdot)\|_{2}^{1-\alpha} \right)^{p-1}\right]\nonumber\nonumber\\
& \times\left(\|\mathcal{L}^{\sigma/2}(u-v)(t,\cdot)\|^{\alpha}_{2}\|(u-v)(t,\cdot)\|_{2}^{1-\alpha} \right)\nonumber\\
&\lesssim \left[\left(\left\|\mathcal{L}^{\sigma/ 2} u(t, \cdot)\right\|_2+\|u(t, \cdot)\|_2\right)^{p-1}+\left(\left\|\mathcal{L}^{\sigma / 2} v(t, \cdot)\right\|_2+\|v(t, \cdot)\|_2\right)^{p-1}\right]\nonumber\\
& \times\left(\left\|\mathcal{L}^{\sigma/ 2}(u-v)(t, \cdot)\right\|_2+\|(u-v)(t, \cdot)\|_2\right).\nonumber
\end{align}
Now, using the above estimate for $v=0$ and the estimates from Theorem \ref{thm:expdecay}  for $u_{0}=0$ and $u_{1}=f(u)$, we get
\begin{multline}
\|\partial_{t}^{i}\mathcal{L}^{j}u^{\operatorname{non}}(t)\|_{2}=\left\|\partial_{t}^{i}\mathcal{L}^{j}\int_{0}^{t}E_{1}(t-\tau)*_{\mathcal{L}}f(u)(t)\mathrm{d}\tau\right\|_{2}\leq \int_{0}^{t}e^{-\delta(t-\tau)}\left\|f(u)(\tau)\right\|_{2}\mathrm{d}\tau\\
    \leq e^{-\delta t}\int_{0}^{t}e^{-\delta (p-1)\tau}\mathrm{d}\tau \|u\|^{p}_{X}\leq e^{-\delta t} \|u\|^{p}_{X}.\nonumber
\end{multline}
Combining the above estimate with \eqref{fineqn0}, we get
\begin{equation}\label{nunorm}
    \|\mathcal{N}u\|_{X}\leq A\|(u_{0},u_{1})\|_{\mathcal{A}^{\sigma,\theta}}+A\|u\|_{X}^p,
\end{equation}
for some positive constant $A$.
\medskip

\medskip

\noindent\textbf{Estimate for $\|\mathcal{N}u-\mathcal{N}v\|_{X}$:} For 
$u, v\in X$, the Definition \ref{nudef} yields
\begin{equation*}
	\mathcal{N}u(t)-\mathcal{N}v(t)=u^{\mathrm{lin}}(t)-v^{\mathrm{lin}}(t)+ u^{\mathrm{non}}(t)-v^{\mathrm{non}}(t).
\end{equation*}
Verifying that $\|u^{\mathrm{lin}}-v^{\mathrm{lin}}\|_{X}=0$ is straightforward.
We must now estimate the following norms in order to calculate the norm $\left\|u^{\operatorname{non}}-v^{\operatorname{non}}\right\|_{X}$:
\begin{equation*}
	\|\partial_{t}^{i}\mathcal{L}^{j}(u^{\operatorname{non}}-v^{\operatorname{non}})(t)\|_{2},\quad (i,j)\in\{(0,0),(0,\sigma/2),(1,0)\}.
\end{equation*}
Utilizing the estimate analogous to the estimate $\left\|\partial_{t}^{i}\mathcal{L}^{j}u^{\operatorname{non}}(t,\cdot)\right\|_{2}$, we get
\begin{align*}
\left\|\partial_{t}^{i}\mathcal{L}^{j}\left(u^{\operatorname{non}}-v^{\operatorname{non}}\right)(t)\right\|_{2} \lesssim  \int_{0}^{t}e^{-\delta(t-\tau)}\|f(u)-f(v)\|_{2}d\tau\lesssim e^{-\delta t}\|u-v\|_{X}\left(\|u\|_{X}^{p-1}+\|v\|^{p-1}_{X}\right).
\end{align*}
Hence, we get
\begin{align*}
\left\|\partial_{t}^{i}\mathcal{L}^{j}\left(u^{\operatorname{non}}-v^{\operatorname{non}}\right)(t)\right\|_{2} \lesssim \|u-v\|_{X}\left(\|u\|_{X}^{p-1}+\|v\|^{p-1}_{X}\right).
\end{align*}
Thus, combining the above estimate with $\|u^{\mathrm{lin}}-v^{\mathrm{lin}}\|_{X}=0$, we have
\begin{equation}\label{numinusnv}
    \|\mathcal{N}u-\mathcal{N}v\|_{X}\leq B\|u-v\|_{X}\left(\|u\|_{X}^{p-1}+\|v\|^{p-1}_{X}\right),
\end{equation}
for some positive constant $B>0$.

Let us set $R:=rA(\|(u_{0},u_{1})\|_{\mathcal{A}^{\sigma,\theta}})$ for some $r>1$ with sufficiently small initial Cauchy data $\|(u_{0},u_{1})\|_{\mathcal{A}^{\sigma,\theta}}< \varepsilon$  such that
\begin{equation*}
    AR^p<\frac{R}{r} \quad \text{and} \quad 2BR^{p-1}<\frac{1}{r}.
\end{equation*}
Then \eqref{nunorm} and \eqref{numinusnv} reduces to  
\begin{align}
	\|\mathcal{N} u\|_{X} \leq  \frac{2R}{r} \quad\text{and } \quad \|\mathcal{N}u-\mathcal{N}v\|_{X} \leq  \frac{1}{r}\|u-v\|_{X},\nonumber
\end{align}  
respectively for all $u, v\in \mathscr{B}_{R}:=\{u\in X: \|u\|_{X}\leq R\}$. As a closed ball of Banach space $X$, $\mathscr{B}_{R}$ is a Banach space  itself. This allows us to conclude the proof of the theorem using the Banach fixed theorem.
\end{proof}

\section{Final remarks}
We wrap up with some observations and notes concerning the results acquired in this study:
\begin{enumerate}
\item The presence of a discrete spectrum enables us to uniformly bound numerous terms, a feature that is generally not available in settings with a continuous spectrum, such as the classical wave equation or the wave equation on the Heisenberg group. This key difference facilitates the derivation of exponential decay estimates.
    \item The decay estimates derived in Theorem \ref{thm:case0new} strengthen those presented in \cite[Proposition 2.1]{palmieri-1}, facilitated by a more refined partition of $\mathcal{I}$ in comparison to \cite{palmieri-1}. 

\item The decay estimates in Theorem \ref{thm:expdecay} remain valid for higher-order spatial and temporal derivatives, allowing the construction of global solutions with enhanced regularity compared to those in Theorem \ref{global:existence:l2}. In this higher-regularity framework, however, the admissible range of parameters becomes more constrained: in certain cases, both the spatial dimension (or, equivalently, the regularity index $\sigma$) and the nonlinearity exponent $p$ are restricted by a  Gagliardo--Nirenberg inequality.
 \item In Theorem \ref{global:existence:l2}, if we drop the assumption that $\mathcal{L}$ is strictly positive operator then we can only expect \textit{local} well-posedness of solution due to lack of decay  observed in Theorem  \ref{thm:case0old}.  After local existence is established, the study naturally turns to long-term behavior, such as whether the solution may be extended worldwide or blows up finitely.
    \end{enumerate}
\section{Acknowledgement}
AT is supported by Institute Post-Doctoral fellowship from Tata Institute of Fundamental Research, Centre For Applicable Mathematics, Bangalore, India.
\bibliographystyle{alphaabbr}
\bibliography{time-fractional}
\end{document}